\documentclass[12pt]{amsart}
\usepackage{ amsmath, amsthm, amsfonts, amssymb, color}
 \usepackage{mathrsfs}
\usepackage{amsfonts, amsmath}
 \usepackage{amsmath,amstext,amsthm,amssymb,amsxtra}
 \usepackage{txfonts} %also pxfonts
 \usepackage[colorlinks, citecolor=blue,pagebackref,hypertexnames=false]{hyperref}
 \allowdisplaybreaks
 \usepackage{pgf,tikz}
 \usepackage{pst-plot}

\begin{document}
 \baselineskip 16.5pt
 \hfuzz=6pt

\newtheorem{theorem}{Theorem}[section]
\newtheorem{proposition}[theorem]{Proposition}
\newtheorem{coro}[theorem]{Corollary}
\newtheorem{lemma}[theorem]{Lemma}
\newtheorem{definition}[theorem]{Definition}
\newtheorem{example}[theorem]{Example}
\newtheorem{remark}[theorem]{Remark}
\newcommand{\ra}{\rightarrow}
\renewcommand{\theequation}
{\thesection.\arabic{equation}}
\newcommand{\ccc}{{\mathcal C}}
\newcommand{\one}{1\hspace{-4.5pt}1}

 \def \Lips  {{   \Lambda}_{L}^{ \alpha,  s }(X)}
\def\BL {{\rm BMO}_{L}(X)}
\def\HAL { H^p_{L,{at}, M}(X) }
\def\HML { H^p_{L, {mol}, M}(X) }
\def\HM{ H^p_{L, {mol}, 1}(X) }
\def\Ma { {\mathcal M} }
\def\MM { {\mathcal M}_0^{p, 2, M, \epsilon}(L) }
\def\dMM { \big({\mathcal M}_0^{p,2, M,\epsilon}(L)\big)^{\ast} }
  \def\RR {  {\mathbb R}^n}
\def\HSL { H^p_{L, S_h}(X) }
\newcommand\mcS{\mathcal{S}}
\newcommand\mcB{\mathcal{B}}
\newcommand\D{\mathcal{D}}
\newcommand\C{\mathbb{C}}
\newcommand\N{\mathbb{N}}
\newcommand\R{\mathbb{R}}
\newcommand\G{\mathbb{G}}
\newcommand\T{\mathbb{T}}
\newcommand\Z{\mathbb{Z}}
\newcommand\lp{L^p(M,\mu)}
\newcommand\Ad{\,{\rm ad}}
\newcommand{\bchi}{\mathlarger{\chi}}
 \newcommand{\comment}[1]{\vskip.3cm
\fbox{%
\color{red}
\parbox{0.93\linewidth}{\footnotesize #1}}}

\newcommand\CC{\mathbb{C}}
\newcommand\NN{\mathbb{N}}
\newcommand\ZZ{\mathbb{Z}}

\renewcommand\Re{\operatorname{Re}}
\renewcommand\Im{\operatorname{Im}}

\newcommand{\mc}{\mathcal}

\def\SL{\sqrt[m] L}
\newcommand{\la}{\lambda}
\def \l {\lambda}
\newcommand{\eps}{\varepsilon}
\newcommand{\pl}{\partial}
\newcommand{\supp}{{\rm supp}{\hspace{.05cm}}}
\newcommand{\x}{\times}
\newcommand{\mar}[1]{{\marginpar{\sffamily{\scriptsize
        #1}}}}
\newcommand{\as}[1]{{\mar{AS:#1}}}

\newcommand\wrt{\,{\rm d}}

\title[ Sharp spectral multipliers without semigroup framework ]
{Sharp spectral multipliers without semigroup framework\\[2pt]   and application to  random walks   }
%generating slowly decaying semigroups}
\author{Peng Chen, \ El Maati Ouhabaz, \ Adam Sikora, \ and \ Lixin Yan}
\address{Peng Chen, Department of Mathematics, Sun Yat-sen (Zhongshan)
University, Guangzhou, 510275, P.R. China}
\email{achenpeng1981@163.com}
\address{El Maati Ouhabaz,  Institut de Math\'ematiques de Bordeaux,  Universit\'e de Bordeaux, UMR 5251,
351, Cours de la Lib\'eration 33405 Talence, France}
\email{Elmaati.Ouhabaz@math.u-bordeaux.fr}
\address{
Adam Sikora, Department of Mathematics, Macquarie University, NSW 2109, Australia}
\email{adam.sikora@mq.edu.au}
\address{
Lixin Yan, Department of Mathematics, Sun Yat-sen (Zhongshan) University, Guangzhou, 510275, P.R. China}
\email{mcsylx@mail.sysu.edu.cn
}

\date{\today}

\begin{abstract}
In this paper we prove  spectral multiplier theorems for abstract self-adjoint operators
 on spaces of homogeneous type. We have two main objectives. The first one is to work outside the semigroup context. In contrast to previous works  on this subject,  we do not make any assumption on the semigroup.  The second  objective is to consider polynomial off-diagonal decay instead of exponential one. Our approach and results lead  to  new applications to several operators such as differential operators, pseudo-differential operators as well as Markov chains. 
 In our general context we introduce a restriction type estimates \`a la Stein-Tomas.  This allows us to obtain 
sharp spectral multiplier theorems and hence sharp Bochner-Riesz  summability results. 
Finally, we consider the random walk on the integer lattice $\Z^n$ and prove  sharp Bochner-Riesz summability  results similar  to those known for the standard Laplacian  on $\R^n$. 

\end{abstract}

\maketitle

 \tableofcontents

\newpage

\section{Introduction}
\setcounter{equation}{0}

Let $(X,d, \mu)$ be a metric measure space, i.e. $X$ is  a metric space with distance function $d$
and   $\mu$ is a nonnegative, Borel, doubling measure on $X$.  Let $A$ be a self-adjoint  operator acting
on $L^2(X,\mu)$. By the spectral theorem one has
$$
A=\int_{-\infty}^\infty \lambda dE(\lambda),
$$
where $dE(\lambda)$ is the spectral resolution of the operator $A$. Then for any bounded measurable
function $F \colon \R  \to \R$ one can define operator
$$
F(A)=\int_{-\infty}^\infty F(\lambda) dE(\lambda).
$$
It is a standard fact that the operator $F(A)$ is  bounded on  $L^2$ with norm  bounded by the $L^\infty$ norm of the function $F$.

The theory of spectral multipliers consists of finding minimal regularity conditions on  $F$ (e.g. existence of a finite number of derivatives of $F$ in a certain  space) which  ensure that the operator $F(A)$
can be extended to a  bounded operator on $L^p(X,\mu)$ for some range
of exponents $p \neq 2$.   Spectral multipliers results are modeled  on Fourier multiplier results described in fundamental works of Mikhlin
\cite{Mi}  and H\"ormander  \cite{Ho4}. The initial motivation for
spectral multipliers  comes from the problem of convergence of Fourier series or more generally of eigenfunction expansion for differential operators.  One of the most famous spectral multipliers is the   Bochner-Riesz mean $$\sigma_{R, \alpha} (A) := \left(1- \frac{A}{R}\right)_+^\alpha.$$ When $\alpha $ is large, the function $\sigma_{R, \alpha}$ is smooth. The problem is then to prove boundedness on $L^p(X)$ (uniformly w.r.t. the parameter $R$) for small values of $\alpha$. This is the reason why, for general function $F$ with compact support, we study $\sup_{t > 0} \| F(tA) \|_{p\to p} \le C < \infty$. The constant $C$ depends on $F$ and measures the (minimal) smoothness required on the function. 

In recent years, spectral multipliers have been studied by many authors in different contexts, including differential or pseudo-differential operators on manifolds, sub-Laplacians on Lie groups, Markov chains as well as operators in abstract settings. We refer the reader to \cite{A1, A2, BO, Blu, C3, CowS, DOS, GHS, Heb1, Heb2, Ho4, JN95, KU, Ma, Mi, SYY} and references therein.  We mention in particular the recent paper \cite{COSY} where sharp spectral multiplier results as well as end-point estimates for Bochner-Riesz means are proved. A restriction type estimate was introduced there in an abstract setting which  turns out to be equivalent to the classical Stein-Tomas restriction estimate in the case of the Euclidean Laplacian. Also it is proved there (see also \cite{BO}) that in an abstract setting, dispersive or Strichartz  estimates for the Schr\"odinger equation imply  sharp spectral multiplier results.

%We study a spectral multipliers of self adjoint operators using the idea
%developed in jesen nakamura. Similar approach has been used in ... and ...
%The most two significant aspects of our approach is to study sharp spectral multipliers via the restriction type estimates and to study %results without making any semigroup assumption.

%Prototype of our assumption is the case of $A=e^{-L}$. This case explain
%our choice of assumption. However, we do not actually assume that
 %$A=e^{-L}$. Our basic assumption is that $A$ is a self-adjoint operator
% but we do not assume that spectrum of $A$ is contain in in the interval $[01]$. We do not even assume that $A$ is positive or bounded. %Therefore we are able to apply our approach to Markov Chains and walks on the graphs.
% Still to understand the origin of our approach it is good to start with
 %the case there $A=\exp(-L)$ for some positive self-adjoint operator $L$.

 \medskip

\subsection{The main results.} There are two main objectives of the present paper. First, in contrast to the previous papers on spectral multipliers where usually decay assumptions are made on the heat kernel or the semigroup, we do not make directly such assumptions and work outside the semigroup framework. The second objective is to replace the usual exponentiel decay of the heat kernel by a polynomial one. All of this is motivated by applications to new settings and examples  which were not covered by previous works. In addition most of spectral multipliers proved before  can be included in our framework.

In order to state explicitly our contributions we first recall that 
$(X, d, \mu)$ satisfies
 the doubling property (see Chapter 3, \cite{CW})
if there  exists a constant $C>0$ such that
\begin{eqnarray}\label{e1.1}
V(x,2r)\leq C V(x, r)\quad \forall\,r>0,\,x\in X.   
\end{eqnarray}
 Note that the doubling property  implies the following
strong homogeneity property,
\begin{equation}
V(x, \lambda r)\leq C\lambda^n V(x,r)
\label{e1.2}
\end{equation}
 for some $C, n>0$ uniformly for all $\lambda\geq 1$ and $x\in X$.
In Euclidean space with Lebesgue measure, the parameter $n$ corresponds to
the dimension of the space, but in our more abstract setting, the optimal $n$
 need not even be an integer. %There also exist $c$ and
%$D, 0\leq D\leq n$ so that
%\begin{equation}
%V(y,r)\leq c\bigg( 1+{d(x,y)\over r}\bigg )^D V(x,r)
%\label{e1.3}
%\end{equation}
%uniformly for all $x,y\in {  X}$ and $r>0$. Indeed, the
%property (\ref{e1.3}) with
%$D=n$ is a direct consequence of triangle inequality of the metric
%$d$ and the strong homogeneity property. In the cases of Euclidean spaces
%${\mathbb R}^n$ and Lie groups of polynomial growth, $D$ can be
%chosen to be $0$.

%In most of the papers mentioned above on spectral multipliers, 
Let $\tau>0$ be a fixed positive parameter and suppose that $A$ is a bounded self-adjoint operator on $L^2(X, \mu)$ which satisfies the following  polynomial off-diagonal decay
\begin{equation}\label{Int-0}
\|P_{B(x,\tau)}{V_{{\tau}}^{\sigma_p}}AP_{B(y,\tau)}\|_{p\rightarrow
2}<C\left(1+\frac{d(x,y)}{\tau}\right)^{-n-a} \quad \forall x,y \in X
\end{equation}
with $\sigma_p={1/p}-{1/2}$ and $P_{B(x,\tau)}$ is the projection on the open ball $B(x,\tau)$. 
We prove that if $a>[n/2]+1$ and $F$ is a bounded Borel function such that
 $F\in  H^s(\mathbb R)$ for some $s>n(1/p-1/2)+1/2$,
then  
$$
\|F(A)A\|_{p\to p} \leq C\|F\|_{H^s}.
$$
Note that the operator $A$ which we discuss here cannot be, in a natural way, considered as a part of a semigroups framework.
See Theorem \ref{th3.1} below for more additional information. 
%\comment{ E.M. Please note that I added "A bounded" in the assumptions. It seems to be  needed in the proofs, 
%especially for  the last lemma in the next subsection}. 
In the particular case where $A = e^{-\tau L}$ for some non-negative (unbounded) self-adjoint operator $L$  with constant in \eqref{Int-0} independent of $\tau$, we obtain for  $s>n(1/p-1/2)+1/2$
	\begin{equation}\label{Int-1}
	\sup_{\tau>0}\|F(\tau L)\|_{p\to p}\leq C\|F\|_{H^s}. 
	\end{equation}
	As mentioned previously, this latter property implies Bochner-Riesz $L^p$-summability with index $\alpha > n(1/p-1/2)$. See Corollary \ref{coro3.2}. 
	
	Some significant spectral multiplier results for operators satisfying polynomial estimates were considered  by Hebisch in
\cite{Heb2} and indirectly also in \cite{JN94, JN95} by Jensen and Nakamura. Our results are inspired by ideas initiated in \cite{DN, Heb2, JN94, JN95, MM}.
 
Following \cite{COSY} we introduce the following  restriction type estimate
\begin{equation}\label{Int-3}
\big\|F(A)AV^{\sigma_p}_{ \tau} \big\|_{p\to 2} \leq C\|F\|_{q}, \ \ \  \ \ \ \ \ \ \ \   \sigma_p= {1\over p}-{1\over 2}.
\end{equation}
We then prove a sharper spectral multiplier result  under this condition. Namely, 
\begin{equation}
\label{Int-2}
 \|AF(  {A})\|_{p\to p} \leq
C_p\|F\|_{W^{s,q}}
\end{equation}
for   $F \in  W^{s,q}({\mathbb R})$ for some
$$ s>
n \left({1\over p}-{1\over 2}\right) +\frac{n}{4[a]}.
$$
 We refer to Theorem \ref{th4.5} for the precise statement.  We prove several other results such as bounds for $e^{itA}A$ on $L^p$ for $p$ as in \eqref{Int-0}. The proofs are very much based on Littlewood-Paley type theory, commutator estimates and amalgam spaces 
 \cite{BDN, DN, JN94}. 
% \comment{ EM. Say more on the strategy of proofs ??}
 
Our result can be applied to many examples. Obviously, if the $A$ has an exponential decay (e.g. a Gaussian upper bound) then it satisfies the previous polynomial off-diagonal decay. Hence our results apply to a wide class of elliptic operators on Euclidean domains or on Riemannian manifolds. They also apply in cases where the heat kernel has polynomial decay. This is the case for example for fractional powers of elliptic or Schr\"odinger operators. In the last section we discuss applications to Markov chains. We also study  spectral multipliers (and hence Bochner-Riesz means) for random walk on $\Z^n$. To be more precise, we consider
\begin{equation*}
Af({\bf {d}}) := \frac1{2n}\sum_{i=1}^{n}\sum_{j=\pm 1} f({\bf {d}}+j{\bf {e}}_i)
\end{equation*}
where ${\bf {e}}_i=(0, \ldots, 1, \ldots 0 )$. We prove for appropriate function $F$ 
\begin{eqnarray*}
\sup_{t>1} \|F(t(I-A)\|_{p \to p} \le C \|F\|_{H^s}
\end{eqnarray*}
for any $s> n|1/p-1/2|$.
If in addition, if 
{\rm supp}~$F\subset [0,1/n]$ and
 \begin{eqnarray*}
 \sup_{1/n>t>0}\|\eta F(t\cdot)\|_{H^s}<\infty
\end{eqnarray*}
for some $s> n|1/p-1/2|$ and $\eta$ is a non-trivial $C_c^\infty(0, \infty)$ function, 
then $F(I-A)$ is bounded on $L^p$ if $1<p< (2n+2)/(n+3)$ and  weak type $(1,1)$ if $p=1$.
This result is similar to the sharp spectral multiplier theorem for the standard Laplacian on $\R^n$. Here again the operator $A$ cannot be, in a natural way,
included in a semigroups framework.

%Our result allows to obtain sharp spectral multipliers with restriction type assumption.
%Sharp Bochner-Riesz summability.  (not possible in AAN approach)

%Almost all spectral multiplier can be reproved this way. {List them COSY, CSY, Hebisch, DOS, Christ, Mauceri-Meda, Alexopoulos.
%inclding restul for

%Significant point - we do not use the semigroup structure  assumption. Surprising this explains the role of semigroup assumption

%If we assume $A=e^{-\sigma L}$ - see Corollary \ref{th3.1}.

%The paper is organized as follows.
%In Section 2 we provide some preliminary results on off-diagonal estimates of the operator $e^{itL}{\varphi} ( 2^{-k}{L})$
%and of  spectral multipliers and Littlewood-Paley theory,  which we need later,
%mainly to prove Proposition~\ref{prop3.3}.
%The  proof of Theorem~\ref{th1.1}  will be given in Section 3. In Section 4 we will apply Theorem~\ref{th1.1}
%to obtain $L^p$-boundedness of the Riesz means of the solutions to the Schr\"odinger equation.

\subsection{ Notations and assumptions}

In the sequel we always assume that the considered ambient space is  a separable metric measure space
  $(X,d,\mu)$ with metric $d$ and Borel measure $\mu$.
We denote by
$B(x,r) :=\{y\in X,\, {d}(x,y)< r\}$  the open ball
with centre $x\in X$ and radius $r>0$.  We   often  use $B$ instead of $B(x, r)$.
Given $\lambda>0$, we write $\lambda B$ for the $\lambda$-dilated ball
which is the ball with the same centre as $B$ and radius $\lambda r$. For $x, \in X$ and $r > 0$ we set $V(x,r) :=\mu(B(x,r))$ the volume of $B(x,r)$.  We set
\begin{eqnarray}\label{e2.1}
V_r(x):=V(x,r), \ \ \ r>0, \ \ x\in X.
\end{eqnarray}
We will often write $V(x)$ in place of $V_1(x).$\\
For $1\le p\le+\infty$, we denote by $\|f\|_p$ the 
norm of  $f\in L^p(X)$,  $\langle .,. \rangle$
the scalar product of $L^2(X)$, and if $T$ is a bounded linear operator from $
L^p(X)$ to $L^q(X)$ we write $\|T\|_{p\to q} $ for its corresponding operator norm.  For a given $p\in [1, 2)$ we define
\begin{eqnarray}\label{e2.2}
 \sigma_p= {1\over p}-{1\over 2}.
\end{eqnarray}
Given a  subset $E\subseteq X$, we  denote by  $\chi_E$   the characteristic
function of   $E$ and  set
$
P_Ef(x)=\chi_E(x) f(x).
$
For every $x\in X$ and $r>0$. 

Throughout this  paper  {we always assume that the space $X$ is of homogeneous type} in the sense that  it satisfies the classical doubling property \eqref{e1.2} with some constants $C$ and $n$ independent of $x \in X$ and $\lambda \ge 1$. 
In the Euclidean space with Lebesgue measure, $n$ is 
the dimension of the space. 
In our results critical index is always expressed in terms of
homogeneous dimension $n$.\\
Note also that there  exists $c$ and $D, 0\leq D\leq n$ so that
\begin{equation}
V(y,r)\leq c\bigg( 1+{d(x,y)\over r}\bigg )^D V(x,r)
\label{e1.3}
\end{equation}
uniformly for all $x,y\in {  X}$ and $r>0$. Indeed, the
property (\ref{e1.3}) with
$D=n$ is a direct consequence of triangle inequality of the metric
$d$ and the strong homogeneity property. In the cases of Euclidean spaces
${\mathbb R}^n$ and Lie groups of polynomial growth, $D$ can be
chosen to be $0$.

\section{Preliminary results}\label{sec2}
\setcounter{equation}{0}
 In this this section we give some elementary results which will be used later. 
\subsection{A criterion for $L^p$-$L^q$ boundedness for linear operators} \label{sec2.1}
We start with a  
countable partitions of $X$. For every $r>0$, we choose  a sequence $(x_i)_{i=1}^{\infty} \in X$ such that
$d(x_i,x_j)> {r/4}$ for $i\neq j$ and $\sup_{x\in X}\inf_i d(x,x_i)
\le {r/4}$. Such sequence exists because $X$ is separable.
Set $$D= \bigcup_{i\in \N}B(x_i,r/4).$$ Then 
define $Q_i(r)$ by the formula
\begin{eqnarray}\label{e2.111}
Q_i(r) =    B(x_i,r/4)\bigcup\left[B\left(x_i,r/2\right)\setminus
\left(\bigcup_{j<i}B\left(x_j,r/2\right)\setminus D\right)\right],
\end{eqnarray}
 so that $\left(Q_i(r)\right)_{i}$ is a countable partition of $X$
 ($Q_i(r)\cap Q_j(r)=\emptyset$ if $i\neq j$).
Note that $Q_i(r)\subset B_i=B(x_i, r)$ and there exists
a uniform constant $C>0$  depending only on the doubling constants in \eqref{e1.2} 
 such that $\mu(Q_i(r))\geq C\mu(B_i)$. 
 %In the sequel, we will write often write $Q_i$ in place of $Q_i(1)$.

We have the following Schur-test for the norm $ \|T\|_{p\to q}$ of  a given linear operator $T$.  
\begin{lemma}\label{le2.2}
Let $T$ be a linear operator   and $1\leq p\leq q\leq \infty$.  For  every $r>0$,
 $$
 \|T\|_{p\to q}\leq \sup_j\sum_i\big\|P_{Q_i(r)}TP_{Q_j(r)}\big\|_{p\to q}+
\sup_i\sum_j\big\|P_{Q_i(r)}TP_{Q_j(r)}\big\|_{p\to q},
 $$
 where   $(Q_i(r))_i$ is  a countable partition of $X$.
 \end{lemma}
\begin{proof}  The proof is inspired by \cite{GHS}.
 Given a function $f\in L^p(X)$, we have
\begin{eqnarray*}
\|Tf\|^q_{q}&= &\big\|\sum_{i,j}P_{Q_i(r)}TP_{Q_j(r)}f\big\|^q_q\\
&=&\sum_i\Big\|\sum_{j}P_{Q_i(r)}TP_{Q_j(r)}f\Big\|^q_q\\
&\leq&\sum_i \left(\sum_j\|P_{Q_i(r)}TP_{Q_j(r)}\|_{p\to q}\|P_{Q_j(r)}f\|_p\right)^q\\
&=:&\big\|\sum_j a_{ij} c_j\big\|_{\ell^q}^q,
\end{eqnarray*}
where $a_{ij}=\|P_{Q_i(r)}TP_{Q_j(r)}\|_{p\to q}$ and $c_j=\|P_{Q_j(r)}f\|_p$.

Next note that,  for all $1 \le  p \le \infty$,
\begin{equation}  \label{ckl}
\sum_k \left(\sum_l |a_{lk} c_l|\right)^p \le \left(\max\left\{ \sup_l
\sum_k |a_{lk}|, \, \sup_k \sum_l |a_{lk}|\right\} \right)^p\sum_n |c_n|^p,
\end{equation}
with the obvious meaning for $p=\infty$, where $c_{lk}$, $a_l$ are sequences of real or complex numbers.
Indeed, for $p=1$ or  $p=\infty$, (\ref{ckl}) is easy to obtain. Then we obtain (\ref{ckl}) for all $1\le p \le \infty$ by interpolation.
Observe that
\begin{eqnarray*}
\|Tf\|_q\leq \|  (a_{ij})\|_{\ell^p\to \ell^q}\|c_j\|_{\ell^p}\le \| (a_{ij})\|_{\ell^p\to \ell^p}\|f\|_{p}.
 \end{eqnarray*}
 The lemma follows from \eqref{ckl} and the above inequality.
\end{proof}

 \subsection{Operators with kernels satisfying  off-diagonal polynomial decays}
%We start with defining the multiplication operator.
 For a given  function $W:X\rightarrow \mathbb{R}$,   we denote by   $M_{W}$ the  multiplication  operator  by $W$, that is
$$
(M_{W}f)(x)=W(x)f(x).
$$
In the sequel, we will  identify  the operator $M_W$ with  the function $W$. This means that, if $T$ is a linear operator,
we will denote by  $W_1TW_2$   the operators  $M_{W_1}TM_{W_2}$. In other words,
$$
W_1TW_2 f(x): =W_1(x)T (W_2f)(x).
$$
 %For every $x\in X$ and $t>0$, we recall that  $V_{t^{1/m}}(x)=V(x, t^{1/m}).$
 Following  \cite{BK1, BK2, BK3},   we  introduce the following  estimates which are interpreted as polynomial off-diagonal estimates.

 %Following  \cite{BK1, BK2, BK3},   let us  introduce the following crucial  estimates which are interpreted as polynomial off-diagonal estimates.
 \begin{definition}
 Let $A$ be a  self-adjoint operator  on $L^2(X)$ and $\tau>0$ be  a  constant.   For $1\leq p<2$ and $a>0$, we say that
$A$ satisfies the property \eqref{pVEp2}  if there exists a constant $C>0$ such that for all  $x,y\in X$,
%We introduce the following  crucial estimates which are interpreted  as  a polynomial off-diagonal estimates for some $t>0,1\leq p<2$ and $a>0$,
\begin{equation*}\label{pVEp2}
\tag{${\rm PVE}^a_{p,2}(\tau)$}\|P_{B(x,\tau)}{V_{{\tau}}^{\sigma_p}}AP_{B(y,\tau)}\|_{p\rightarrow
2}<C\left(1+\frac{d(x,y)}{\tau}\right)^{-n-a} 
\end{equation*}
with $\sigma_p={1/p}-{1/2}.$
\end{definition}
%\comment{EM: I removed the hypothesis " A is non-negative" in the definition, I think it is not needed. }

\smallskip

By  H\"older's inequality and duality, the condition \eqref{pVEp2} implies  that
\begin{eqnarray}\label{e2.223}
 \|P_{B(x,\tau)}AP_{B(y,\tau)}\|_{p\to p} + \|P_{B(x,\tau)}AP_{B(y,\tau)}\|_{p'\to p'}
 \leq C \left(1+\frac{d(x,y)}{\tau}\right)^{-n-a}.
\end{eqnarray}
By interpolation,
\begin{eqnarray} \label{e2.3}
 \|P_{B(x,\tau)}AP_{B(y,\tau)}\|_{2\to 2}
 \leq C \left(1+\frac{d(x,y)}{\tau}\right)^{-n-a}.
\end{eqnarray}

\begin{remark}
Suppose that \eqref{pVEp2} holds for some $p<2$.  Then  ${\rm (PVE^a_{q,2}) }$   holds for every $q \in [p, 2].$
This can be shown by applying  complex interpolation to the family
$$F(z) = P_{B(x,\tau)}{V_{{\tau}}^{z \, \sigma_p }}AP_{B(y,\tau)}.$$
For $z= 1+is$ we  use  \eqref{pVEp2} and for $z = is$ we use \eqref{e2.3}.
\end{remark}

%\comment{EM: I added this second  remark since it will be used latter.}

In the sequel,  
for a given $\tau>0$   we fix a countable partition $ \{Q_\ell(\tau)\}_{\ell=1}^{\infty} $
 of $X$  and and  a sequence $(x_{\ell})_{\ell=1}^{\infty} \in X$  as in Section 2.1. First, we have 
 the following result.

\begin{proposition}\label{prop2.1} Let    $ 1\leq p\leq 2$,
and $A$ be a self-adjoint operator on $L^2(X)$. Assume that condition \eqref{pVEp2}
holds for some $\tau>0$ and $a>0$.  There exists a constant $C>0$ independent of $\tau>0$ such that
\begin{eqnarray}\label{e2.e1}
\big\|V_{\tau}^{{\sigma_p}}A\big\|_{p\to 2}\leq C.
\end{eqnarray}
As a consequence, the operator $A$ is a bounded operator on $L^p(X)$, and
 $
\|A\|_{p\to p}+\|A\|_{2\to 2}\leq C.
 $
\end{proposition}

\begin{proof} By Lemma~\ref{le2.2} and condition \eqref{pVEp2},   one is lead to estimate
\begin{eqnarray}\label{e2.e2}
  \sup_j \sum_i \left(1 + \frac{d(x_i, x_j)}{\tau} \right)^{-n-a}\leq C<\infty.
\end{eqnarray}
Note that for every   $k\geq 1$, 
\begin{eqnarray} \label{qq}
 \sup_{j, \, \tau}\# \big\{i:   2^k \tau \leq d(x_i, x_j)< 2^{k+1}\tau \big\}
 &\le&
%  \sup_{j, \, \tau} \sup_{ \{i:   2^k \tau \leq d(x_i, x_j)< 2^{k+1}\tau  \}}
%   {V(x_j, 2^{k+1}\tau  )  \over   V(x_i, \tau)}\nonumber\\
%    &\le&
  \sup_{j, \, \tau} \sup_{ \{i:   2^k \tau \leq d(x_i, x_j)< 2^{k+1}\tau  \}}
   {V(x_i, 2^{k+3}\tau  )  \over   V(x_i, \tau)}\nonumber\\
    &\le& C 2^{kn}< \infty.
\end{eqnarray}
This implies that  for every $j\geq 1$ and $k\geq 1$,
$$
\sum_{i:\, 2^{k}\tau \leq d(x_i, x_j)<2^{k+1} \tau} \left(1 + \frac{d(x_i, x_j)}{\tau} \right)^{-n-a}
\leq C (1+2^k)^{-n-a}2^{kn}
 \leq  C 2^{-ka}$$
and we sum over $k$ to get \eqref{e2.e2}. 

 The boundedness of the operator $A$ on $L^p$ is proved in the same way by applying \eqref{e2.223}.
\end{proof}

\smallskip

Note that when the operator $A$ has integral kernel $K_A(x,y)$
satisfying the following pointwise estimate 
$$ 
|K_A(x,y)|\leq  C V(x, \tau)^{-1}\left(1+\frac{d(x,y)}{\tau}\right)^{-n-a}
$$
 for some $\tau>0$ and all $x, y\in X$, then $A$ satisfies the property \eqref{pVEp2} with $p=1$. 
 Conversely, we have the following result.
 
 \begin{proposition} \label{prop2.5} Suppose that $a>D$ where $D$ is the constant
 	in \eqref{e1.3}. 
If the operator $A$ satisfies  the property \eqref{pVEp2} for some $\tau>0$ and $p=1$, then   the operator $A^2$ has integral kernel $K_{A^2}(x,y)$
satisfying the following pointwise estimate: For any $\epsilon>0$, there exists a constant $C>0$ independent of $\tau$ such that
 \begin{eqnarray} \label{e2.e3}
|K_{A^2}(x,y)|\leq  C V(x, \tau)^{-1}\left(1+\frac{d(x,y)}{\tau}\right)^{-(a-D-\epsilon)} 
 \end{eqnarray} 
 for all $x, y\in X$.
 \end{proposition}

 \begin{proof}  For every  $x,y\in X$ and $\tau>0$, we write
  \begin{eqnarray*}
 P_{B(x, \tau)} A^2 P_{B(y, \tau)} 
= P_{B(x, \tau)} AP_{B(x, \tau)} A P_{B(y, \tau)}  +\sum_{\ell=0}^{\infty} \sum_{i: 2^{\ell}\tau \leq d(x_i, x)< 2^{\ell+1}\tau}
 P_{B(x, \tau)} A  P_{B(x_i, \tau)} 
 A  P_{B(y, \tau)}.   
 \end{eqnarray*}  
From  the property \eqref{pVEp2}  with $p=1$  we get 
 \begin{eqnarray*}
 \left\| P_{B(x, \tau)} A^2 P_{B(y, \tau)}\right\|_{L^1\to L^{\infty}} 
 &\leq& C V(x, \tau)^{-1 }  
  \left(1+\frac{d(x ,y)}{\tau}\right)^{-n-a}\\
  & &+ C\sum_{\ell=0}^{\infty} \sum_{i: 2^{\ell}\tau \leq d(x_i, x)< 2^{\ell+1}\tau} V(x_i, \tau)^{-1} 
  \left(1+\frac{d(x,x_i)}{\tau}\right)^{-n-a}\left(1+\frac{d(x_i,y)}{\tau}\right)^{-n-a}.   
 \end{eqnarray*} 
 This, in combination with the fact that $V(x_i, \tau)^{-1} \leq C2^{\ell D} V(x, \tau)^{-1} $ for every $\ell\geq 0 $ and 
 $2^{\ell}\tau \leq d(x_i, x)< 2^{\ell+1}\tau,$   and the property \eqref{qq},  implies that
 \begin{eqnarray*}
 \left\| P_{B(x, \tau)} A^2 P_{B(y, r)}\right\|_{L^1\to L^{\infty}} 
 %&\leq&  C \sum_{\ell=0}^{\infty} 2^{-\ell \epsilon}V(x, \tau)^{-1 }   \left(1+\frac{d(x ,y)}{\tau}\right)^{-(a-2D-\epsilon)}\\
 &\leq&   C V(x, \tau)^{-1 }  
  \left(1+\frac{d(x ,y)}{\tau}\right)^{-(a-D-\epsilon)} 
 \end{eqnarray*}
 for any $\epsilon>0$.
 Hence   it follows   that  \eqref{e2.e3} holds. 
 This completes the proof of Proposition~\ref{prop2.5}. 
 \end{proof}

 Finally, we mention that if $A$ is the semigroup $e^{-t L}$ generated by  (minus) a non-negative self-adjoint 
 operator $L$, then the condition  ${\rm PVE}^a_{p,2}(t^{1/m})$ holds for some $m\geq 1$ if the corresponding heat kernel $K_t(x,y)$ has a polynomial decay
\begin{equation}\label{e-poly}
|K_{t}(x,y)|\leq  C V(x, t^{1/m})^{-1}\left(1+\frac{d(x,y)}{t^{1/m}}\right)^{-n-a}.
\end{equation}
  it is known that the heat kernel satisfies a Gaussian upper bound, for a wide class of differential operators of order $ m \in 2 \NN$ on Euclidean domains
   or Riemannian manifolds (see for example \cite{Dav89}).
   In this case \eqref{e-poly} holds with any arbitrary  $a > 0$.

\bigskip

\section{Spectral multipliers   via  polynomial  off-diagonal decay kernels }
\setcounter{equation}{0}

 In this section  we  prove   spectral multiplier results corresponding to compactly supported  functions
  in the abstract setting of self-adjoint operators  on homogeneous spaces.
  Recall that we assume 
   that  $(X, d, \mu)$ is  a  metric measure  space satisfying the doubling property
   and  $n$ is the   homogeneous  dimension from  condition \eqref{e1.2}.
   We use the standard notation $H^{s}(\R) $ for the Sobolev space
$\|F\|_{H^{s}}  = \|(I-d^2/dx^2)^{s/2}F\|_{2}$.

\begin{theorem}\label{th3.1}
Suppose that $A$ is a bounded self-adjoint operator on $L^2(X)$
which satisfies condition \eqref{pVEp2} for some $\tau>0$, $1\leq p<2$ and  $a>[n/2]+1$.
If $F$ is a bounded Borel function such that
 $F\in  H^s(\mathbb R)$ for some $s>n(1/p-1/2)+1/2$,
then  there exists constant $C>0$ such that
$$
\|F(A)A\|_{p\to p} \leq C\|F\|_{H^s}.
$$
The constant $C$ above does not depend on the choice of $\tau>0$. 
In addition, if we assume that  $A$ is a nonnegative  self-adjoint operator  and $\supp F\subset [-1, 1]$, then   there exists a constant $C>0$ 
which is also independent of $\tau$  such that
$$
\|F(-\log A)\|_{p\to p}\leq C\|F\|_{H^s}.
$$
\end{theorem}

%\comment{ E.M. Please note that I added "A bounded" in the assumptions. It seems to be  needed in the proofs, 
%especially for  the last lemma in the next subsection}. 

The proof of  Theorem~\ref{th3.1} will be given at the end of this section after a series of preparatory results. The following statement is a direct consequence of  Theorem~\ref{th3.1}.
% We mention the following corollary in which $A$ is the semigroup $e^{-tL}$ at time $t = \tau$. 
%In this case,  ${\rm PVE}^a_{p,2}(t^{1/m})$ holds with $p = 1$ for  all $t > 0$ if the heat kernel of $L$ 
%satisfies the polynomial decay \eqref{e-poly}.

\begin{coro}\label{coro3.2}
	Suppose that $L$ is a non-negative self-adjoint operator on $L^2(X)$ that for a given $\tau>0$ the semigroup operator 
	$e^{-\tau L}$ satisfies condition \eqref{pVEp2} for some  $1\leq p<2$ and  $a>[n/2]+1$.
	If $F$ is a bounded Borel function such that $\supp F\subset [-1, 1]$ and
	$F\in  H^s(\mathbb R)$ for some $s>n(1/p-1/2)+1/2$,
	then  there exists constant $C>0$  such that
	$$
	\|F(\tau L)\|_{p\to p}\leq C\|F\|_{H^s}.
	$$
	As a consequence  if operators $e^{-\tau L}$ satisfies condition  \eqref{pVEp2} with constant independent of $\tau$  then
	$$
	\sup_{\tau>0}\|F(\tau L)\|_{p\to p}\leq C\|F\|_{H^s}
	$$
	for the same range of exponents $p$ and $s$. 
\end{coro}
\begin{proof}
	We apply  Theorem \ref{th3.1} to the operator $A=e^{-\tau L}$
	for all $\tau>0$. Then the theorem  follows from  the fact that the constants in 
	statement of Theorem \ref{th3.1} are independent of $\tau$. 
\end{proof}

\begin{remark}
	It is possible to obtain a version of Theorem \ref{th3.1} 
	under the weaker assumption  $a>0$. However, this requires a different 
	approach which we do not discuss here.  Related results  can be found  in \cite{Heb2} and \cite{Ma}. We will consider this more general case in a forthcoming paper \cite{COSY3}. 
\end{remark}

\medskip

\subsection{Preparatory  results}
The  following   result  plays  a essential role in the  proof of    Theorem~\ref{th3.1}.
 \begin{theorem}\label{le3.4}
Suppose that $A$   is a bounded self-adjoint operator on $L^2(X)$
and satisfies condition \eqref{pVEp2} for $1\leq p<2$, $a>[n/2]+1$ and for some $\tau>0$.
 Then there exists a constant $C>0$ such that for all $  t\in{\mathbb R},$
 \begin{eqnarray}\label{e3.5}
\|e^{it A}A\|_{p\to p}\leq C(1+|t|)^{n\sigma_p}
\end{eqnarray}
where $\sigma_p=1/p-1/2$.
%where $\varphi \in C_c^\infty(0,\infty)$.
\end{theorem}

\begin{remark}
	It is interesting to compare the above statement with Theorem 1.1 of \cite{DN}. Note that in Theorem~\ref{le3.4} we assume  condition \eqref{pVEp2} for one fixed exponent $\tau$ but conclusion is valid for all $t\in \R$.  
\end{remark}

In order to prove Theorem~\ref{le3.4} we need two technical lemmas. First,
we  state the following known formula for the commutator of a Lipschitz function and an operator
$T$ on $L^2(X)$ on metric measure space $X$.  Recall our notation convention $ \eta T= M_\eta T$.

\begin{lemma}\label{le3.2}
Let $T$ be a self-ajoint operator on $L^2(X)$. Assume that for  some $\eta\in {\rm Lip}(X)$, the
commutator   $[\eta, T]$, defined by $[\eta, T]f=M_{\eta} Tf-TM_{\eta} f$,   satisfies
that for all $f\in {\rm Dom}(T)$, $\eta f\in {\rm Dom}(T)$ and
$$
\|[\eta, T]f\|_2\leq C\|f\|_2,\ \ \ % \ \ \ \forall f\in {\rm Dom}(T) \cap {\rm Dom}(TM_\eta),
$$
where  ${\rm Dom}(T)$ denotes the domain of $T$ .
 Then the following formula holds:
\begin{eqnarray*}
[\eta, e^{itT}]f=it \int_0^1 e^{istT}[\eta, T]  e^{i(1-s)tT}f ds\ \ \ \ \forall t\in{\mathbb R}, \forall f\in L^2(X).
\end{eqnarray*}
\end{lemma}

\begin{proof}
The proof of Lemma~\ref{le3.2}  follows   by integration by parts. See for example,   \cite[Lemma 3.5]{MM}.
\end{proof}

Next we recall  some useful results concerning amalgam spaces \cite{BDN, DN, JN94}.
For a given  $\tau>0$, we recall that  $ \{Q_\ell(\tau)\}_{\ell=1}^{\infty} $
is  a countable partition of $X$ as in Section \ref{sec2.1}.
For $1\leq p, q\leq \infty$ and a non-negative number $\tau>0$, consider   a two-scale Lebesgue space $X_\tau^{p,q}$
of measurable functions on $X$ equipped with the norm
\begin{eqnarray}\label{e3.1}
\|f\|_{X_\tau^{p,q}}:=\left( \sum_{\ell=1}^{\infty}\|P_{Q_\ell(\tau)}f\|_{q}^p\right)^{1/p}.
\end{eqnarray}
(with obvious changes if $q=\infty$).

 Notice that when $q=p$ these spaces are just the Lebesgue spaces $X_\tau^{p, p}=L^p$ for  every $\tau>0$ and $p.$
% Since $\ell^{p_1}\subset \ell^{p_2}$ for $p_1<p_2$
% one has the inclusion
Note also that for $1\leq   p_1<p_2\leq \infty,$ we have that $X_\tau^{p_1,q}\subseteq X_\tau^{p_2,q}$ with
\begin{eqnarray*}%\label{e3.111}
 \|f\|_{X_\tau^{p_2,q}}\leq  \|f\|_{X_\tau^{p_1,q}}.
\end{eqnarray*}

%\comment{Amalgama space - quote Cowling and Terence Tao? P. Auscher? Tent space.}

The following result  gives a  criterion for a linear operator $A$ to be bounded on   $X_\tau^{p, 2}, 1\leq p\leq 2.$
We define a family of functions  $\{\eta_\ell(x)\}_{\ell}$  by
\begin{eqnarray}\label{e12.3}
\eta_\ell(x): ={d(x,x_{\ell})\over \tau}, \ \ \ \   \ell=1, 2, \cdots 
\end{eqnarray}
 i.e., the distance function between $x_{\ell}$ and $x$ (up to a constant).

\begin{lemma}\label{le3.3} Let $\tau>0$ and $\{\eta_\ell(x)\}_{\ell=1}^{\infty}$ be as above. For a bounded operator $T$ on $L^2(X)$, the multi-commutator
  ${\rm ad}_{\ell}^k(T): L^2(X)\to L^2(X) $
is defined inductively by
\begin{eqnarray}\label{e12.3}
{\rm ad}_{\ell}^0(T)=T, \ \ \   {\rm ad}_{\ell}^k(T)={\rm ad}_{\ell}^{k-1}({\eta_{\ell}}T-T{\eta_{\ell}}), \ \ \ \ \ell\geq 1, \    k\geq 1.
\end{eqnarray}
Suppose  that there exists a constant $M>1$ such that
for all $1\leq \ell<\infty,$
$$
\|{\rm ad}_{\ell}^k(T)\|_{2\to 2}\leq CM^k, \ \ \ \  0\leq k\leq [n/2]+1.
$$
Then for  $1\leq p\leq 2$,
$$
\|V_\tau^{{\sigma_p}}TV_\tau^{ {-\sigma_p}} f\|_{X_\tau^{p,2}}\leq CM^{n\sigma_p}\|f\|_{X_\tau^{p,2}} \  \ \
{\rm with}\ \sigma_p=(1/p-1/2)
$$
for some constant  $C>0$ depending on $n, p$ and $\|T\|_{2\to 2}.$
\end{lemma}

\begin{proof}
We  prove this lemma  by using  the complex interpolation method.
Let  ${\bf S}$ be  the closed strip $0\leq {\rm Re} z\leq 1$
 in the complex plane. For every $z\in {\bf S}$, we consider the analytic family
of  operators:
$$
T_z :=  V_r^{z/ 2}\, T \, V_\tau^{-{z/ 2}}.
$$

Consider $z=1+iy,\,y\in \mathbb{R}$ and $p=1$. Let  ${N}\geq 1$ be a constant large enough  to be chosen later.
Let $\{Q_\ell(\tau)\}_\ell$ be  a countable partition of $X$ in Section  2.1. By definition of $X_\tau^{1,2},$
\begin{eqnarray}\label{e3.2}
\|T_{1+iy} f\|_{X_\tau^{1,2}}&=&\sum_{j=1}^{\infty}\|P_{Q_j(\tau)}V_\tau^{\frac{1+iy}{2}}T V_\tau^{-\frac{1+iy}{2}}f\|_2\nonumber\\
&\leq & \sum_{\ell=1}^{\infty}\sum_{j=1}^{\infty}
\|P_{Q_j(\tau)}V_\tau^{\frac{1+iy}{2}}T V_\tau^{-\frac{1+iy}{2}}P_{Q_\ell(\tau)}f\|_2\nonumber\\
&=& \left(\sum_{\ell } \sum_{j: \, d(x_{\ell},x_j)>{N}\tau}
       + \sum_{\ell } \sum_{j: \,  d(x_{\ell},x_j)\leq N \tau} \right)
      \|P_{Q_j(\tau)}V_\tau^{\frac{1+iy}{2}}TV_\tau^{-\frac{1+iy}{2}}P_{Q_\ell(\tau)}f\|_2\nonumber\\
	   &=:&I+II.
	    \end{eqnarray}
By  the Cauchy-Schwarz  inequality, we  obtain 	
\begin{eqnarray}\label{e3.3}
II&\leq& C\sum_{\ell=1}^{\infty}
\left(\sum_{j: \, d(x_{\ell},x_j)\leq {N}r}V_\tau(x_j)\right)^{1/2}
\left(\sum_{j:\, d(x_{\ell},x_j)\leq {N}\tau}\|P_{Q_j(\tau)}TV_\tau^{-\frac{1+iy}{2}}P_{Q_\ell(\tau)}f\|^2_2\right)^{1/2} \nonumber \\
%&\leq& C\sum_{\ell=1}^{\infty}
%\left(\sum_{j: \, d(x_{\ell},x_j)\leq {N}r}\mu(Q_j(r))\right)^{1/2}
%\left(\sum_{j:\, d(x_{\ell},x_j)\leq {N}r}\|P_{Q_j(r)}TV_r^{-\frac{1+iy}{2}}P_{Q_\ell(r)}f\|^2_2\right)^{1/2} \nonumber \\
 &\leq& C  \sum_{\ell } V(x_{\ell}, {N}\tau)^{1/2}  \| TV_\tau^{-\frac{1+iy}{2}}P_{Q_\ell(\tau)}f\|_2  \nonumber \\
 &\leq& C{N}^{n/2}\|T\|_{2\to 2} \sum_{\ell }\mu(Q_\ell(\tau))^{1/2}
   \| V_\tau^{-\frac{1+iy}{2}} P_{Q_\ell(\tau)}f\|_2   \nonumber\\
  &\leq& C{N}^{n/2}\|T\|_{2\to 2}\|  f\|_{X_\tau^{1,2}}.
\end{eqnarray}
Now we estimate the term $I.$ Let $\kappa=[n/2]+1$. We  apply the Cauchy-Schwarz  inequality again to obtain
\begin{eqnarray}\label{e3.4}
I&\leq& C \sum_{\ell=1}^{\infty}  \left(\sum_{j: \, d(x_{\ell},x_j)>{N}\tau}\mu(Q_j(\tau))\eta_{\ell}(x_j)^{-2{\kappa}}\right)^{1/2}
\left(  \sum_{j:d(x_{\ell},x_j)>{N}\tau}\eta_{\ell}(x_j)^{2\kappa} \|P_{Q_j(\tau)}T V_\tau^{-\frac{1+iy}{2}}P_{Q_{\ell}(\tau)}f\|^2_2\right)^{1/2}\nonumber\\
 &\leq& C\sum_{\ell }\left(\sum_{k=0}^{\infty}
 \sum_{j:\, 2^{k}{N}\tau <  d(x_{\ell},x_j)\leq 2^{k +1}{N}\tau} \mu(Q_j(\tau))\eta_{\ell}(x_j)^{-2\kappa}\right)^{1/2}
 \|\eta_{\ell}^{{\kappa}}TV_\tau^{-\frac{1+iy}{2}}P_{Q_{\ell}(\tau)}f\|_2\nonumber\\
 &\leq& C\sum_{\ell }\left(\sum_{k=0}^{\infty} (2^{k} {N})^n \mu(Q_{\ell}(\tau)) (2^{k} {N})^{-2\kappa} \right)^{1/2}
 \|\eta_{\ell}^{{\kappa}}TV_\tau^{-\frac{1+iy}{2}}P_{Q_{\ell}(\tau)}f\|_2\nonumber\\
 &\leq& C\sum_{\ell}{N}^{-\kappa+n/2} \mu(Q_{\ell}(\tau))^{1/2}
 \|\eta_{\ell}^{\kappa}TV_\tau^{-\frac{1+iy}{2}}P_{Q_{\ell}(\tau)}f\|_2.
\end{eqnarray}
 %\comment{EM:  Here (and at other places) left multiplication by powers of $V$ disappeared !!  This is ok if it is bounded above and below on each ball of radius $1$. I'm missing something here ? Check again.}
To continue we define  we let  $\Gamma(\kappa,0)=1$ for $\kappa\geq 1$, and $\Gamma(\kappa, m)$ defined  inductively by
 $\Gamma(\kappa, m+1)=\sum_{\ell=m}^{\kappa-1}\Gamma(\ell, m)$ for $1\le m \le \kappa-1$.
Applying  the following   known formula for commutators (see Lemma 3.1, \cite{JN95}):
 $$
\eta_{\ell}^{\kappa} T =\sum_{m=0}^{\kappa} \Gamma(\kappa, m) {\rm ad}^m_\ell(T) \eta_{\ell}^{\kappa-m},
$$
we obtain
\begin{eqnarray}\label{comutator}
\|\eta_{\ell}^{\kappa}T(1+\eta_{\ell})^{-\kappa} \|_{2\to 2}
 &\leq& C\sum_{m=0}^{\kappa} \|{\rm ad}^m_{\ell}(T) \eta_{\ell}^{\kappa-m}(1+\eta_{\ell})^{-\kappa} \|_{2\to 2}\nonumber
 \\&\leq& C\sum_{m=0}^{\kappa} \|{\rm ad}^m_{\ell}(T) \|_{2\to 2}\leq CM^{\kappa}. 
\end{eqnarray}
This implies 
\begin{eqnarray*}
I&\leq&  C{N}^{-\kappa+n/2} M^{\kappa} \sum_{\ell} \mu(Q_{\ell}(\tau))^{1/2}
 \|V_\tau^{-\frac{1+iy}{2}}(1+\eta_{\ell})^{\kappa}P_{Q_{\ell}(\tau)}f\|_2\\
 &\leq&  C{N}^{-\kappa+n/2} M^{\kappa} \sum_{\ell }
 \|P_{Q_{\ell}(\tau)}f\|_2\leq C{N}^{-\kappa+n/2} M^{\kappa} \|f\|_{X_\tau^{1,2}}.
\end{eqnarray*}
Next, set ${N}=M\|T\|_{2\to 2}^{-1/\kappa}$. Then  above estimates of $I$ and $II$ give 
\begin{eqnarray}\label{i1}
\|T_{1+iy} f\|_{X_\tau^{1,2}}&\leq& C \left({N}^{-\kappa+n/2} M^{\kappa}
+ {N}^{n/2}\|T\|_{2\to 2}\right)
\|f\|_{X_\tau^{1,2}}\nonumber\\
 &\leq& C M^{n/2}\|T\|_{2\to 2}^{1-1/\kappa} \|f\|_{X_\tau^{1,2}}. 
\end{eqnarray}
On the other hand, we have that for $z=iy,\,y\in\mathbb{R}$,
\begin{eqnarray}\label{i2}
\|T_{iy}\|_{X_\tau^{2,2}\to X_\tau^{2,2}}= \|T_{iy}\|_{2\to 2}\leq \|T\|_{2\to 2}
\end{eqnarray}
From  estimates \eqref{i1} and \eqref{i2}, we apply the complex interpolation method  to   obtain  
\begin{eqnarray*}
\|V_\tau^{{\sigma_p}}TV_\tau^{ {-\sigma_p}} f\|_{X_\tau^{p,2}}=\|T_{\frac{2}{p}-1} f\|_{X_\tau^{p,2}}
  \leq C  M^{n\sigma_p}\|f\|_{X_\tau^{p,2}}, \ \ \ \ \sigma_p=(1/p-1/2) 
\end{eqnarray*}
   for some constant  $C>0$ depending on $n, p$ and $\|T\|_{2\to 2}.$
This finishes the proof  of Lemma~\ref{le3.3}.
\end{proof}

 \medskip
 
Now we apply Lemmas~\ref{le3.2} and ~\ref{le3.3}  to prove    Theorem~\ref{le3.4}.

\begin{proof}[Proof f Theorem~\ref{le3.4}]    The  proof    is
inspired  by  Theorem 1.3 of \cite{JN95} and Theorem 1.1 of  \cite{DN}.
Note that
$$
\|e^{itA}A\|_{p\to p}\leq
\|V_{\tau}^{{-\sigma_p}}\|_{X_{\tau}^{p,2}\to L^p}\|V_{\tau}^{{\sigma_p}}e^{itA}
V_{\tau}^{{-\sigma_p}}\|_{X_{\tau}^{p,2}\to X_{\tau}^{p,2}}
\|V_{\tau}^{{\sigma_p}}A\|_{L^p \to X_{\tau}^{p,2}}.
$$

First, we have that  $\|V_{\tau}^{{-\sigma_p}}\|_{X_{\tau}^{p,2}\to L^p}\leq C$ by definition of $X_{\tau}^{p,2}$
and H\"older's  inequality.
To estimate the term $\|V_{\tau}^{{\sigma_p}}A\|_{L^p \to X_{\tau}^{p,2}}$, we recall that  $(\{Q_i(\tau)\}_{i=1}^{\infty})$
is  a countable partition of $X$ as in Section 2.1 and note that
\begin{eqnarray*}
\|V_{\tau}^{{\sigma_p}}Af\|_{X_{\tau}^{p,2}}&=&
\left(\sum_{j=1}^{\infty}\|P_{Q_j(\tau)}V_{\tau}^{{\sigma_p}}Af\|_2^p\right)^{1/p}\\
&\leq& \left(\sum_j\left[\sum_\ell\|P_{Q_j(\tau)}V_{\tau}^{{\sigma_p}}AP_{Q_\ell(\tau)}\|_{p\to2}
\|P_{Q_\ell(\tau)}f\|_p\right]^p\right)^{1/p}.
\end{eqnarray*}
For every $j, \ell$, we set $a_{j\ell}=\|P_{Q_j(\tau)}V_{\tau}^{{\sigma_p}}AP_{Q_{\ell}(\tau)}\|_{p\to 2}$
and $b_\ell=\|P_{Q_{\ell}(\tau)}f\|_p$. It follows by interpolation that 
\begin{eqnarray*}
\|V_{\tau}^{{\sigma_p}}Af\|_{X_{\tau}^{p,2}}&\leq& \big\|\sum_\ell a_{j\ell}b_\ell\big\|_{\ell^p}\\
&\leq& \big\|(a_{j\ell})\big\|_{\ell^p\to \ell^p}\|b_\ell\|_{\ell^p}\\
&\leq&  \big\|(a_{j\ell})\big\|_{\ell^1\to \ell^1}^\theta \big\|(a_{j\ell})
\big\|_{\ell^\infty\to \ell^\infty}^{1-\theta} \|f\|_{p}.
\end{eqnarray*}
 Therefore, by condition \eqref{pVEp2}  we have
$$
 \|(a_{j\ell})\|_{\ell^1\to \ell^1}=
 \sup_\ell\sum_{j}\big\|P_{Q_j(\tau)}V_{\tau}^{{\sigma_p}}AP_{Q_{\ell}(\tau)}\big\|_{p\to 2}\leq C
$$
and
$$
 \|(a_{j\ell})\|_{\ell^\infty\to \ell^\infty}=\sup_j\sum_{\ell}
 \big\|P_{Q_j(\tau)}V_{\tau}^{{\sigma_p}}AP_{Q_{\ell}(\tau)}\big\|_{p\to 2}\leq C.
$$
Thus
 \begin{equation}\label{eq001}
 \big\|V_{\tau}^{{\sigma_p}}A\big\|_{L^p \to X_{\tau}^{p,2}}\leq C.
 \end{equation}
Next we  show that
\begin{eqnarray}\label{e3.6}
\big\|V_{\tau}^{{\sigma_p}}e^{itA}V_{\tau}^{{-\sigma_p}}\big\|_{X_{\tau}^{p,2}\to X_{\tau}^{p,2}}\leq C(1+|t|)^{n\sigma_p}.
\end{eqnarray}
 By Lemma~\ref{le3.3}, it suffices to show for every $m\in{\mathbb N},$
\begin{eqnarray}\label{e3.61}
\|{\rm ad}_m^k(e^{itA})\|_{2\to 2}\leq C(1+|t|)^k, \,\,\, 0\leq k\leq [n/2]+1.
\end{eqnarray}
Note that $A$ is a bounded operator on $L^2(X)$. Recall that (cf. Lemma~\ref{le3.2})
 for all $m\in{\mathbb N}$:
\begin{eqnarray*}
{\rm ad}_m^1(e^{itA})f= it\int_0^1 e^{istA} {\rm ad}_m^1(A) e^{i(1-s)tA}fds, \ \ \ \ \forall t\in{\mathbb R},
\ \ f\in L^2(X).
\end{eqnarray*}
Repeatedly, we are reduced to prove that for every $m\in{\mathbb N},$ there exists a constant $C>0$
independent of $m$ such that,
\begin{eqnarray}\label{e3.62}
\|{\rm ad}_m^k(A)\|_{2\to 2}\leq C, \,\,\, 0\leq k\leq [n/2]+1.
\end{eqnarray}
Fix $m\in {\mathbb N}$.
By Lemma~\ref{le2.2}, it suffices to show
\begin{eqnarray}\label{e3.7}
\sup_j\sum_\ell\|P_{Q_{\ell}(\tau)}{\rm ad}_{m}^k(A)P_{Q_j(\tau)}\|_{2\to 2}\leq C, \,\,\, 0\leq k\leq [n/2]+1
\end{eqnarray}
for some constant $C>0$ independent of $m.$

To show  \eqref{e3.7} we note that
\begin{eqnarray*}
P_{Q_{\ell}(\tau)}{\rm ad}_{m}^k(A)P_{Q_j(\tau)}&=& \sum_{\alpha+\beta+\gamma=k}\frac{k!}{\alpha !\,  \beta ! \, \gamma !}
[\eta_{m}(x_\ell)-\eta_{m}(x_j)]^\beta
[\eta_{m}(\cdot)-\eta_{m}(x_\ell)]^\alpha \times \\
&& \hspace{2cm}\times  P_{Q_{\ell}(\tau)}AP_{Q_j(\tau)}
[\eta_{m}(x_j)-\eta_{m}(\cdot)]^\gamma.
\end{eqnarray*}
Observe  that:

\smallskip

$\bullet$  $ |\eta_{m}(x_\ell)-\eta_{m}(x_j)|\leq d(x_\ell,x_j)/\tau
 $;

\smallskip

$\bullet$ 
  $|\eta_{m}(x)-\eta_{m}(x_\ell)|\chi_{Q_{\ell}(\tau)}\leq d(x,x_\ell)\chi_{Q_{\ell}(\tau)}/\tau\leq 1$; 
  
\smallskip

$\bullet$ 
$|\eta_{m}(x_j)-\eta_{m}(y)|\chi_{Q_j(\tau)}\leq d(y,x_j)\chi_{Q_j(\tau)}/\tau\leq 1$.

\noindent
These, together with estimate \eqref{pVEp2} and $a>[n/2]+1\geq k$, yield 
$$
\sum_\ell\|P_{Q_{\ell}(\tau)}{\rm ad}_{m}^k(A)P_{Q_j(\tau)}\|_{2\to 2}\leq
C \sum_\ell \left(1+ {d(x_\ell,x_j)\over  \tau}  \right)^k \|P_{Q_{\ell}(\tau)}AP_{Q_j(\tau)}\|_{2\to 2}
\leq C
$$
for some constant $C>0$ independent of $j$. Hence,   \eqref{e3.7}
is proved. This, in combination with estimates \eqref{e3.61} and \eqref{e3.62},
 implies  \eqref{e3.6}. All together, we obtain that $\|e^{itA}A\|_{p\to p}\leq C(1+|t|)^{n\sigma_p}
$
where $\sigma_p=(1/p-1/2)$. The proof of Theorem~\ref{le3.4} is complete.
\end{proof}

\medskip

\subsection{Proof of Theorem~\ref{th3.1}}
We apply Lemma~\ref{le3.4} to see that
\begin{eqnarray*}
\|F(A)A\|_{p\to p}&\leq& \int_{\mathbb{R}} |\widehat{F}(\xi)| \|e^{i\xi A}A\|_{p\to p}\,d\xi\\
&\leq& C\int_{\mathbb{R}} |\widehat{F}(\xi)| (1+|\xi|)^{n\sigma_p}\,d\xi\\
&\leq&C \|F\|_{H^s} \left(\int_{\mathbb{R}} (1+|\xi|)^{2(n\sigma_p-s)}\,d\xi\right)^{1/2}\\
&\leq& C\|F\|_{H^s}.
\end{eqnarray*}

If we also assume that $A\geq 0$ and $\supp F\subset [-1, 1]$, then
we may consider  $G(\lambda) :=F(-\log \lambda)\lambda^{-1}$ so that  $F(-\log A)=G(A)A$.  Therefore, 
\begin{eqnarray*}
\|F(-\log A)\|_{p\to p}&\leq& \|G(A)A\|_{p\to p}\leq C\|G\|_{H^s}.
\end{eqnarray*}
Since  $\supp F\subset [-1,1]$ we have   $ \|G\|_{H^s}\leq C\|F\|_{H^s}$ and we obtain 
$\|F(-\log A)\|_{p\to p} \leq C\|F\|_{H^s}.
$
The proof of  Theorem~\ref{th3.1} is complete.
\hfill{}$\Box$

\medskip

\section{Sharp spectral multiplier  results   via  restriction type estimates}
\setcounter{equation}{0}

 The aim of this section is to  obtain sharp $L^p$ boundedness of  spectral multipliers from  restriction type estimates. We consider the metric measure space  $(X, d, \mu)$ with satisfies the doubling condition \eqref{e1.2}
 with the  homogeneous dimension $n$. 
 Let $q\in [2, \infty]$. Recall that $V_r = V(x,r) = \mu(B(x,r))$.
We say that  $A$ satisfies the
    {\it  restriction type condition}   if there exist interval $[b,e]$ for some $-\infty<b<e<\infty$ and $\tau>0$ such that
  for  all Borel functions $F$ with $\supp F\subset [b,e]$,
\begin{equation*}\label{st}
\big\|F(A)AV^{\sigma_p}_{ \tau} \big\|_{p\to 2} \leq C\|F\|_{q}, \ \ \  \ \ \ \ \ \ \ \   \sigma_p= {1\over p}-{1\over 2}.
\tag{$ {{\rm  ST^q_{p, 2}}(\tau) }$}
\end{equation*}
The above conditions originates and in fact is a version of the classical 
Stein-Tomas restriction estimates. For more detailed discussion and  the rationale of formulation of the above condition we referee 
readers to  \cite{COSY, SYY}) for a related definition).

The following statement is our  main result in this section.

 \begin{theorem}\label{th4.5} Let $1\leq p <2\leq q\leq  \infty$. Let  $A$ be
 a bounded self-adjoint operator on $L^2(X)$ satisfying  the property   $\eqref{pVEp2}$  for some 
  $a>[n/2]+1$ and $\tau>0$. 
  Suppose also that $A$ satisfies the property  $\eqref{st}$  on  some interval $[R_1,R_2]$
  for  $-\infty<R_1< R_2<\infty$.  
  Let  $F$ be  a bounded Borel  function such that  $\supp F\subset [R_1+\gamma,  R_2-\gamma]$  for some $\gamma > 0$ and $F \in  W^{s,q}({\mathbb R})$
for some
$$ s>
n \left({1\over p}-{1\over 2}\right) +\frac{n}{4[a]}.
$$
Then  $AF( A)$ is bounded on $L^p(X)$ %for all $ p \in (p_0, 2]$ 
and
\begin{equation}
\label{e4.6}
 \|AF(  {A})\|_{p\to p} \leq
C_p\|F\|_{W^{s,q}}.
\end{equation}
\end{theorem}

\begin{remark}
	One can formulate a version of Corollary \ref{coro3.2} corresponding to
	Theorem \ref{th4.5}. See also the statement of Theorem \ref{coro6.2}.  
\end{remark}

\begin{remark}\label{rem01}
Note that if $A$ satisfies $\eqref{pVEp2}$  for some 
  $a>[n/2]+1$ and $\tau>0$ then $A$ satisfies $\eqref{st}$ with $q = \infty$. Indeed,
by Proposition \ref{prop2.1}, we have
\begin{eqnarray*}
\big\|F(A)AV^{\sigma_p}_{ \tau} \big\|_{p\to 2} &\leq& \|F (A)\|_{2\to 2} \| AV^{\sigma_p}_{ \tau} \big\|_{p\to 2}\\
&\le& C \|F\|_\infty. 
\end{eqnarray*}
As a consequence of Theorem \ref{th4.5} we obtain under the sole assumption  $\eqref{pVEp2}$
\begin{equation}
\label{ee4.6}
 \|AF(  {A})\|_{p\to p} \leq
C_p\|F\|_{W^{s,\infty}}
\end{equation}
for every $F \in W^{s,\infty}$ and some $s>
n \left({1\over p}-{1\over 2}\right) +\frac{n}{4[a]}$. 
\end{remark}
%\comment{E.M. I have three comments : 1) I changed the conclusion of the theorem into boundedness on $L^p$ instead of boundedness on $L^s$ for  $s \in (p, 2]$. I have the impression that what the proof gives is the statement above. Please check again !!\\
%2) I think we need "$A$ bounded" here since the proof  uses Theorem 3.1. \\
%3) Should we state here a version of Corollary 3.2 with $F(tL)$  for possibly unbounded $L$ ?  }

\medskip
Before we start the proof of  Theorem~\ref{th4.5},  we need some preliminary result. 
For a given  $r>0$, we recall that  $(Q_i(r))_i$ 
is a countable partition of $X.$
For a  bounded operator $T$ and a given $r>0$,  we decompose the operator $T$
 into the on-diagonal part   $[T]_{<r}$  and the off-diagonal part  $[T]_{>r}$  as follows:
\begin{eqnarray}\label{e4.m1}
[T]_{<r}:=\sum_i\sum_{j:d(x_i,x_j)\leq  5r} P_{Q_j(r)}TP_{Q_i(r)} 
\end{eqnarray}
and
\begin{eqnarray}\label{e4.m2}
[T]_{>r}:=\sum_i\sum_{j:d(x_i,x_j)>  5r} P_{Q_j(r)}TP_{Q_i(r)}.
\end{eqnarray}
 For the on-diagonal part $[T]_{<r}$, we have the following result.

\begin{lemma}\label{ondiagnal}
Assume that $T$ is a bounded operator from $L^p(X)$ to $L^q(X)$ ($p\leq q$). 
 Then   the on-diagonal part $[T]_{<r}$
is   bounded on from $L^p(X)$ to $L^q(X)$ and there exists a constant  $C=C(n)>0$ independent of $r$ such that 
$$
\|[T]_{<r}\|_{p\to q}\leq C\|T\|_{p\to q}.
$$
\end{lemma}

\begin{proof} Note that 
\begin{eqnarray*}
 \left\|\sum_{i}\sum_{j:d(x_i,x_j)\leq 5r} P_{Q_j(r)}TP_{Q_i(r)}f\right\|^q_{q} 
&&= \sum_{j}\Big\|\sum_{i:d(x_i,x_j)\leq 5r} P_{Q_j(r)}TP_{Q_i(r)}f\Big\|^q_{q}\\
&&\leq C\sum_{j}\sum_{i:d(x_i,x_j)\leq 5r} \Big\|P_{Q_j(r)}TP_{Q_i(r)}f\Big\|^q_{q}\\
%&&\leq C\sum_{i}\sum_{j:d(x_i,x_j)\leq 5w}\Big\| P_{Q_j(w)}TP_{Q_i(w)}f\Big\|^q_{q}\\
%&&\leq C\sum_{i}\sum_{j:d(x_i,x_j)\leq 5w}\| P_{Q_j(w)}T\|_{p\to q}^q\|P_{Q_i(w)}f\|^q_{p}\\
&&\leq C\sum_{i}\sum_{j:d(x_i,x_j)\leq 5r}\| T\|_{p\to q}^q\|P_{Q_i(r)}f\|^q_{p}\\
&&\leq C\| T\|_{p\to q}^q\sum_{i}\|P_{Q_i(r)}f\|^q_{p}\\
&&\leq C\| T\|_{p\to q}^q\|f\|^q_{p}.
\end{eqnarray*}
This  proves Lemma~\ref{ondiagnal}.
\end{proof}

\medskip

\begin{proof}[Proof of Theorem~\ref{th4.5}]
%Given  $\gamma>0$, we set $d'(x,y) := d(x,y)/\gamma^{1/m}$,
 %we see that  ($PVE^a_{p,2}(1)$) and  ${\rm (ST_{p, 2})}$ hold for  $A$ and $d'$  with same constant $C$ and $a$. Therefore,
%we can assume that $A$  satisfies ($PVE^a_{p,2}(1)$) and  ${\rm (ST_{p, 2})}$ for $\gamma=1$.
Let   $\phi \in C_c^{\infty}(\mathbb R) $   be  a function such that $\supp \phi\subseteq \{ \xi: 1/4\leq |\xi|\leq 1\}$ and
 $
\sum_{\ell\in \ZZ} \phi(2^{-\ell} \lambda)=1$ for all
${\lambda>0}.
$
Set $\phi_0(\lambda)= 1-\sum_{\ell=1}^{\infty} \phi(2^{-\ell} \lambda)$. By the Fourier inversion formula, we can write
\begin{eqnarray}\label{e4.11}
 AF(A) =\sum_{\ell=0}^{\infty} AF^{(\ell)}({A}),
\end{eqnarray}
where
\begin{eqnarray} \label{es1}
F^{(0)}(\lambda)=\frac{1}{2\pi}\int_{-\infty}^{+\infty}
 \phi_0(s) \hat{F}(s) e^{is\lambda} \;ds
\end{eqnarray}
and
\begin{eqnarray} \label{es2}
F^{(\ell)}(\lambda) =\frac{1}{2\pi}\int_{-\infty}^{+\infty}
 \phi(2^{-\ell}s) \hat{F}(s) e^{is\lambda} \;ds.
\end{eqnarray}

We now estimate $L^p$-$L^p$ norm of the operator $ AF^{(\ell)}({A}), \ell\geq 0.$
Let $N\geq 1$ be    a constant   to be chosen later. For every $\ell\geq 0$,  we follow \eqref{e4.m1} and \eqref{e4.m2} to write 
\begin{eqnarray}\label{e4.99}
AF^{(\ell)}({A})=:[AF^{(\ell)}({A})]_{<N\tau}+[AF^{(\ell)}({A})]_{>N\tau}.
\end{eqnarray}

\bigskip

\noindent
{\underline{Estimate of  the term $[AF^{(\ell)}({A})]_{<N\tau}$.}} \ 
Let   $\psi \in C_c^{\infty}(\mathbb R) $   be  a function such that  $\psi(\lambda)=1$ for $\lambda\in [R_1+\gamma/2, R_2-\gamma/2]$ and $\supp \psi\subset [R_1,R_2]$. We write
$$
[AF^{(\ell)}({A})]_{<N\tau}=[\psi(A)AF^{(\ell)}({A})]_{<N\tau}+
[(1-\psi(A))AF^{(\ell)}({A})]_{<N\tau}.
$$
Observe that
\begin{eqnarray*}
 \left\|[\psi(A)AF^{(\ell)}({A})]_{<N\tau}\right\|^p_{p} 
 &&\leq C
 \sum_{j}\sum_{\lambda: d(x_j, x_\lambda)\leq 5N\tau}\| P_{Q_\lambda(N\tau)}A\psi(A) F^{(\ell)}({A})P_{Q_j(N\tau)}f\|^p_{p}.
 %\\
 %&&\leq C
 %\sum_{j}\sum_{\lambda: d(x_j, x_\lambda)\leq 5N\tau}
 %\mu(Q_\lambda(N\tau))^{{\sigma_p}}\| P_{Q_\lambda(N\tau)}\psi(A)AF^{(\ell)}({A})P_{Q_j(N\tau)}f\|_{2}\\
 %&&\leq C
 %\sum_{j}\sum_{\lambda: d(x_j, x_\lambda)\leq 5N\tau}\mu(Q_j(N\tau))^{{\sigma_p}}\|F^{(\ell)}\|_q \|V_{\tau}^{1/2-1/p}P_{Q_j(N\tau)}f\|_{p}\\
 \end{eqnarray*}
We use the H\"older inequality to obtain 
\begin{eqnarray*}
\| P_{Q_\lambda(N\tau)}A\psi(A) F^{(\ell)}({A})P_{Q_j(N\tau)}f\|_{p} 
&\leq&  
 \mu(Q_\lambda(N\tau))^{{\sigma_p}}\| P_{Q_\lambda(N\tau)}A\psi(A) F^{(\ell)}({A})P_{Q_j(N\tau)}f\|_{2}\\
 &\leq&  C
 \mu(Q_j(N\tau))^{{\sigma_p}}\|F^{(\ell)}\|_q \|V_{\tau}^{1/2-1/p}P_{Q_j(N\tau)}f\|_{p}.
\end{eqnarray*}
Note that in the last inequality we used the condition $d(x_j,x_\lambda)\leq 5N\tau$ and 
 the fact that   $\supp \psi F^{(\ell)}\subset [R_1,R_2]$, 
 and it follows from  the property  \eqref{st}   that
 $
\|A \psi(A) F^{(\ell)}({A})V_{\tau}^{\sigma_p}\|_{p\to 2} \leq \|F^{(\ell)}\|_q.
 $
Hence,  
\begin{eqnarray}\label{kkl}
 \left\|[\psi(A)AF^{(\ell)}({A})]_{<N\tau}\right\|^p_{p} 
&&\leq  CN^{np\sigma_p}\|F^{(\ell)}\|^p_q  \sum_{j}\mu(Q_j(\tau))^{p\sigma_p}\|V_{\tau}^{-\sigma_p}P_{Q_j(N\tau)}f\|^p_{p}\nonumber\\
&&\leq CN^{np\sigma_p} \|F^{(\ell)}\|^p_q\sum_{j}\|P_{Q_j(N\tau)}f\|^p_{p}\nonumber\\
&&\leq CN^{np\sigma_p}\|F^{(\ell)}\|^p_q\|f\|^p_{p}.
\end{eqnarray}
On the other hand, we apply Lemma~\ref{ondiagnal} and Theorem~\ref{th3.1} to obtain that  for some $s>n/2+1$ and every large number $M>0$,
\begin{eqnarray*}
\|[(1-\psi(A))AF^{(\ell)}({A})]_{<N\tau}\|_{p\to p}&\leq& C \|(1-\psi(A))AF^{(\ell)}({A})\|_{p\to p}\\
&\leq &C\|(1-\psi(\cdot))F^{(\ell)}(\cdot)\|_{H^s} 
 \leq  C_{\sigma,M}2^{-M\ell}\|F\|_q.
\end{eqnarray*}
This, together with \eqref{kkl},  yields 
\begin{eqnarray}\label{eaa}
\|[AF^{(\ell)}({A})]_{<N\tau}f\|^p_{p}
\leq C N^{np\sigma_p}\Big[\|F^{(\ell)}\|^p_q+2^{-Mp\ell}\|F\|^p_q\Big]\|f\|^p_{p}.
%\\&&\leq CN^{n(1-p/2)}\|F^{(\ell)}\|^p_q\|f\|^p_{p}.
\end{eqnarray}

\medskip

\noindent
{\underline{Estimate of  the term $[AF^{(\ell)}({A})]_{\geq N\tau}$.}} \  From \eqref{es1} and \eqref{es2}, 
we have that
\begin{eqnarray}\label{nnn}
[AF^{(\ell)}({A})]_{>N\tau}=\frac{1}{2\pi}\int_{-\infty}^{+\infty}
 \phi(2^{-\ell}t) \hat{F}(t) [Ae^{it A}]_{>N\tau} \;dt.
\end{eqnarray}
To estimate the term $ [Ae^{it A}]_{>N\tau}$, we write
\begin{eqnarray*}
Ae^{it A}
%&=&[e^{i\tau A}]_{>N\tau/2}[A]_{>N\tau/2}\\
%& & +\left([e^{i\tau A}]_{>N\tau/2}[A]_{<N\tau/2}+[e^{i\tau A}]_{<N\tau/2}[A]_{>N\tau/2}\right)\\
%& &+[e^{i\tau A}]_{<N/2}[A]_{<N\tau/2}\\
&=&e^{it A} [A]_{>N\tau/2} 
 +[e^{it A}]_{>N\tau/2}[A]_{<N\tau/2} 
  + [e^{it A}]_{<N\tau/2}[A]_{<N\tau/2}\\
 &=:&T_1+T_2+T_3.
\end{eqnarray*}
It is easy to see that
$$
\sum_{i}\sum_{j:d(x_i,x_j)> 5N\tau} P_{Q_j(N\tau)}T_3P_{Q_i(N\tau)}=0,
$$
and so 
$$
[Ae^{it A}]_{>N\tau} =:[T_1]_{>N\tau}
+[T_2]_{>N\tau}.
$$
By Lemma~\ref{ondiagnal} with $r=N\tau$, we obtain
\begin{eqnarray} \label{et0}
\|[Ae^{it  A}]_{>N\tau}\|_{p\to p}&\leq& \|T_1-[T_1]_{<N\tau}\|_{p\to p}+\|T_2-[T_2]_{<N\tau}\|_{p\to p}\nonumber\\
&\leq& (1+C)\|T_1\|_{p\to p}+(1+C)\|T_2\|_{p\to p}, 
\end{eqnarray}
and hence we have to estimate $\|T_i\|_{p\to p}, i=1,2.$

Let us estimate the term $\|T_1\|_{p\to p}$. We write
\begin{eqnarray}\label{et4}
\|T_1\|_{p\to p}&=&\|e^{itA}[A]_{>N\tau/2}\|_{p\to p}\nonumber\\
&\leq&
\left\|V_{\tau}^{{-\sigma_p}}\right\|_{X_{\tau}^{p,2}\to L^p}
\left\|V_{\tau}^{{\sigma_p}}e^{itA}
V_{\tau}^{{-\sigma_p}}\right\|_{X_{\tau}^{p,2}\to X_{\tau}^{p,2}}
\left\|V_{\tau}^{{\sigma_p}}[A]_{>N\tau/2}\right\|_{L^p \to X_{\tau}^{p,2}}.
\end{eqnarray}
We have that  $\|V_{\tau}^{{-\sigma_p}}\|_{X_{\tau}^{p,2}\to L^p}\leq C$ by definition of $X^{p,2}$
and H\"older's  inequality. Also it follows from \eqref{e3.6}  that
\begin{eqnarray}\label{e33.6}
\left\|V_{\tau}^{{\sigma_p}}e^{itA}V_{\tau}^{{-\sigma_p}}\right\|_{X_{\tau}^{p,2}\to X_{\tau}^{p,2}}\leq C(1+|t|)^{n\sigma_p}.
\end{eqnarray}
To handle the term 
$
\big\|V_{\tau}^{{\sigma_p}}[A]_{>N\tau/2}\big\|_{L^p \to X_{\tau}^{p,2}},
$
we note that
\begin{eqnarray*}
\left\|V_{\tau}^{{\sigma_p}}[A]_{>N\tau/2}f\right\|_{X_{\tau}^{p,2}}&=&
\left(\sum_{j=1}^{\infty}\|P_{Q_j(\tau)}V_{\tau}^{{\sigma_p}}[A]_{>N\tau/2}f\|_2^p\right)^{1/p}\\
&\leq& \left(\sum_j\left[\sum_{\ell:d(x_\ell,x_j)>N\tau/2}
\left\|P_{Q_j(\tau)}V_{\tau}^{{\sigma_p}}AP_{Q_\ell(\tau)}\right\|_{p\to2}
\left\|P_{Q_\ell(\tau)}f\right\|_p\right]^p\right)^{1/p}.
\end{eqnarray*}
For every $j, \ell$, we set $a_{j\ell}=\|P_{Q_j(\tau)}V_{\tau}^{{\sigma_p}}AP_{Q_\ell(\tau)}\|_{p\to 2}$
and $b_\ell=\|P_{Q_\ell(\tau)}f\|_p$. This gives
\begin{eqnarray*}
\|V_{\tau}^{{\sigma_p}}[A]_{>N\tau/2}f\|_{X_{\tau}^{p,2}}&\leq& \Big\|\sum_\ell a_{j\ell}b_\ell\Big\|_{\ell^p}\\
&\leq& \Big\|(a_{j\ell})\Big\|_{\ell^p\to \ell^p}\|b_\ell\|_{\ell^p}\\
&\leq&  \Big\|(a_{j\ell})\Big\|_{\ell^1\to \ell^1}^\theta \Big\|(a_{j\ell})
\Big\|_{\ell^\infty\to \ell^\infty}^{1-\theta} \|f\|_{p}
\end{eqnarray*}
by  interpolation.
It then follows from condition~\eqref{pVEp2}  that
$$
 \|(a_{j\ell})\|_{\ell^1\to \ell^1}=
 \sup_\ell\sum_{j:d(x_\ell,x_j)>N\tau/2}\Big\|P_{Q_j(\tau)}V_{\tau}^{{\sigma_p}}AP_{Q_\ell(\tau)}\Big\|_{p\to 2}\leq CN^{-a}
$$
and
$$
 \|(a_{j\ell})\|_{\ell^\infty\to \ell^\infty}=\sup_j\sum_{\ell:d(x_\ell,x_j)>N\tau/2}
 \Big\|P_{Q_j(\tau)}V_{\tau}^{{\sigma_p}}AP_{Q_\ell(\tau)}\Big\|_{p\to 2}\leq CN^{-a},
$$
and so 
 \begin{equation*} 
 \Big\|V_{\tau}^{{\sigma_p}}[A]_{>N\tau/2}\Big\|_{L^p \to X_{\tau}^{p,2}}\leq CN^{-a}.
 \end{equation*}
This, in combination with \eqref{e33.6} and \eqref{et4}, shows that
 \begin{eqnarray}\label{et1}
\|T_1\|_{p\to p}&=&\|e^{itA}[A]_{>N\tau/2}\|_{p\to p} \leq CN^{-a}(1+t)^{n\sigma_p}.  
\end{eqnarray}
  
Next we estimate   the term $\|T_2\|_{p\to p}$.   We write
\begin{eqnarray}\label{et21}
\|T_2\|_{p\to p}&=&\|[e^{itA}]_{>N\tau/2}[A]_{<N\tau/2}\|_{p\to p}\nonumber\\
&\leq&
\left\|V_{\tau}^{{-\sigma_p}}\right\|_{X_{\tau}^{p,2}\to L^p}
\left\|V_{\tau}^{{\sigma_p}}[e^{itA}]_{>N\tau/2}
V_{\tau}^{{-\sigma_p}}\right\|_{X_{\tau}^{p,2}\to X_{\tau}^{p,2}}
\left\|V_{\tau}^{{\sigma_p}}[A]_{<N\tau/2}\right\|_{L^p \to X_{\tau}^{p,2}}.
\end{eqnarray}
We have that  $\|V_{\tau}^{{-\sigma_p}}\|_{X_{\tau}^{p,2}\to L^p}\leq C$ by definition of $X^{p,2}$
and H\"older's  inequality.  Also, we follow a similar argument as in \eqref{eq001}  to obtain  
\begin{eqnarray}\label{et22}
\|V_{\tau}^{{\sigma_p}}[A]_{<N\tau/2}\|_{L^p \to X_{\tau}^{p,2}}\leq C.
\end{eqnarray}
In order to deal with  the term $\|V_{\tau}^{{\sigma_p}}[e^{itA}]_{>N\tau/2}
V_{\tau}^{{-\sigma_p}}\|_{X_{\tau}^{p,2}\to X_{\tau}^{p,2}}$, we  apply the complex interpolation method.
To do it, we  let  ${\bf S}$ denote the closed strip $0\leq {\rm Re} z\leq 1$
 in the complex plane. For $z\in {\bf S}$, we consider an analytic family
of  operators:
$$
T_z=  V_{\tau}^{z\over 2}[e^{itA}]_{>N\tau/2} V_{\tau}^{-{z\over 2}}.
$$

\smallskip

\noindent
{\bf Case 1:}  $z=1+iy,\,y\in \mathbb{R}$ and $p=1$.\
 In this case, we observe that
\begin{eqnarray}\label{e3.2}
\|T_{1+iy} f\|_{X_{\tau}^{1,2}}&=&\sum_{j=1}^{\infty}\big\|P_{Q_j(\tau)}V_{\tau}^{\frac{1+iy}{2}}
[e^{itA}]_{>N\tau/2} V_{\tau}^{-\frac{1+iy}{2}}f\big\|_2\nonumber\\
&\leq & \sum_{\ell=1}^{\infty}\sum_{j=1}^{\infty}
\big\|P_{Q_j(\tau)}V_{\tau}^{\frac{1+iy}{2}}[e^{itA}]_{>N\tau/2} V_{\tau}^{-\frac{1+iy}{2}}
P_{Q_{\ell}(\tau)}f\big\|_2\nonumber\\
&\leq& \sum_{\ell=1}^{\infty} \sum_{j: \, d(x_\ell,x_j)>N\tau/2}
      \big\|P_{Q_j(\tau)}V_{\tau}^{\frac{1+iy}{2}}e^{itA}V_{\tau}^{-\frac{1+iy}{2}}
      P_{Q_{\ell}(\tau)}f\big\|_2.
	  %\nonumber\\
%&=:& \sum_{\ell=1}^{\infty} E_{\ell}.
	  \end{eqnarray}
By the Cauchy-Schwarz  inequality,
\begin{eqnarray*} 
 &&\hspace{-0.8cm} \sum_{j: \, d(x_\ell,x_j)>N\tau/2}
      \big\|P_{Q_j(\tau)}V_{\tau}^{\frac{1+iy}{2}}e^{itA}V_{\tau}^{-\frac{1+iy}{2}}
      P_{Q_{\ell}(\tau)}f\big\|_2\\
	  &&\leq \left(\sum_{j: \, d(x_{\ell},x_j)>N\tau/2}\mu(Q_j(\tau))\eta_{\ell}(x_j)^{-2{[a]}}\right)^{1/2}
\left(  \sum_{d(x_{\ell},x_j)>N\tau/2}\eta_{\ell}(x_j)^{2{[a]}} 
\left\|P_{Q_j(\tau)}e^{itA} V_{\tau}^{-\frac{1+iy}{2}}P_{Q_{\ell}(\tau)}f\right\|^2_2\right)^{1/2},
\end{eqnarray*}
where $\eta_{\ell}(x):=d(x,x_\ell)/\tau$.  
Since
\begin{eqnarray*} 
\sum_{j: \, d(x_{\ell},x_j)>N\tau/2}\mu(Q_j(\tau))\eta_{\ell}(x_j)^{-2{[a]}}
 &\leq& C \sum_{i=0}^{\infty}
 \sum_{j:\, 2^{i}N\tau <  d(x_{\ell},x_j)\leq 2^{i +1}N\tau} \mu(Q_j(\tau))\eta_{\ell}(x_j)^{-2{[a]}} \nonumber\\
 &\leq& C \sum_{i=0}^{\infty} (2^{i} N)^n \mu(Q_{\ell}(\tau)) (2^{i} N)^{-2{[a]}}  \nonumber\\
 &\leq& C N^{-2{[a]}+n } \mu(Q_{\ell}(\tau))  
\end{eqnarray*}
and
\begin{eqnarray*}
 \sum_{j:\, d(x_{\ell},x_j)>N\tau/2}\eta_{\ell}(x_j)^{2{[a]}} 
\big\|P_{Q_j(\tau)}e^{itA} V_{\tau}^{-\frac{1+iy}{2}}P_{Q_{\ell}(\tau)}f\big\|^2_2  
\leq C\big\|\eta_{\ell}^{{[a]}}e^{itA}V_{\tau}^{-\frac{1+iy}{2}}P_{Q_{\ell}(\tau)}f\big\|_2^2, 
\end{eqnarray*}
we have    
\begin{eqnarray*} 
 \|T_{1+iy} f\|_{X_{\tau}^{1,2}}
 &\leq& C\sum_{\ell=1}^{\infty}N^{-{[a]}+n/2} \mu(Q_{\ell}(\tau))^{1/2}
 \big\|\eta_{\ell}^{{[a]}}e^{itA}V_{\tau}^{-\frac{1+iy}{2}}P_{Q_{\ell}(\tau)}f\big\|_2 
\end{eqnarray*}

Following  an argument as in  \eqref{comutator} and \eqref{e3.61} we  obtain 
\begin{eqnarray*}
\|\eta_{\ell}^{{[a]}}e^{itA}(1+\eta_{\ell})^{-{[a]}} \|_{2\to 2}
\leq C(1+t)^{[a]}.
\end{eqnarray*}
From this, it follows  that
\begin{eqnarray*}
\|T_{1+iy} f\|_{X_{\tau}^{1,2}}&\leq&  CN^{-{[a]}+n/2} (1+t)^{[a]}\sum_{i } \mu(Q_{\ell}(\tau))^{1/2}
 \left\|V_{\tau}^{-\frac{1+iy}{2}}(1+\eta_{\ell})^{{[a]}}P_{Q_{\ell}(\tau)}f\right\|_2\\
 &\leq&  CN^{-{[a]}+n/2} (1+t)^{[a]} \sum_{i }
 \|P_{Q_{\ell}(\tau)}f\|_2\\
 &\leq& CN^{-{[a]}+n/2} (1+t)^{[a]}\|f\|_{X_{\tau}^{1,2}}.
\end{eqnarray*}

\smallskip

\noindent
{\bf Case 2:} $z=iy,\,y\in \mathbb{R}$ and $p=2$.\
In this case, we note that
 $$
\left\|T_{iy}\right\|_{X_{\tau}^{2,2}\to X_{\tau}^{2,2}}= \|T_{iy}\|_{2\to 2}\leq \|[e^{itA}]_{>N\tau/2}\|_{2\to 2}
\leq (1+C)\|e^{itA}\|_{2\to 2}\leq 1+C.
 $$

 \medskip
 
 From {\bf Cases 1} and {\bf  2}, we apply the complex interpolation method to obtain
\begin{eqnarray*}
\left\|V_{\tau}^{{\sigma_p}}[e^{itA}]_{>N\tau/2}V_{\tau}^{ {-\sigma_p}} f\right\|_{X_{\tau}^{p,2}}
=\left\|T_{\frac{2}{p}-1} f\right\|_{X_{\tau}^{p,2}}
  \leq C_p \left(N^{-{[a]}+n/2} (1+t)^{[a]}\right)^{\theta} \left\|f\right\|_{X_{\tau}^{p,2}}
\end{eqnarray*}
where $\theta=2/p-1$.
This implies
\begin{eqnarray*}
\|T_2\|_{p\to p}\leq C N^{(-{[a]}+n/2)(2/p-1)} (1+t)^{{[a]}(2/p-1)} 
\end{eqnarray*}
Substituting  this and estimate  \eqref{et1}  back into \eqref{et0}, yields
\begin{eqnarray*}
 \|[Ae^{it A}]_{>N\tau}\|_{p\to p}
&\leq &C\|T_1\|_{p\to p}+C\|T_2\|_{p\to p} 
%&\leq &CN^{-a}(1+t)^{n(1/p-1/2)}+C(N^{(-{[a]}+n/2)(2/p-1)} (1+t)^{{[a]}(2/p-1)})\\
\leq  C N^{(-{[a]}+n/2)(2/p-1)} (1+t)^{{[a]}(2/p-1)}.
\end{eqnarray*}
This, in combination with \eqref{nnn} and  \eqref{eaa},  shows  
\begin{eqnarray*}
\|AF^{\ell}(A)\|_{p\to p}&\leq& CN^{n(1/p-1/2)}\big[\| F^{(\ell)}(\lambda)\|_{q}+2^{-M\ell}\|F\|_q\big]\\
&+&
CN^{(-{[a]}+n/2)(2/p-1)} \int_{2^\ell}^{2^{\ell+1}}
|\phi(2^{-\ell}t)| | \hat{F}(t)|(1+t)^{{[a]}(2/p-1)} \;dt.
\end{eqnarray*}
We take $N=2^{\ell(1+p/(2{[a]}(2-p)))}$ to obtain
$$
\|AF^{\ell}(A)\|_{p\to p}\leq C2^{\ell n(1/p-1/2)(1+p/(2{[a]}(2-p)))}\left(\| F^{(\ell)}(\lambda)\|_{q}+2^{-M\ell}\|F\|_q+\|\phi_{\ell}\widehat{F}\|_2\right).
$$
After summation in  $\ell$,  we obtain
\begin{eqnarray*} 
\|AF(A)\|_{p\to p}\leq C\|F\|_{W^{s, q}}
\end{eqnarray*}
 as   in \eqref{e4.6}, whenever
  $$s>
 n \left({1\over p}-{1\over 2}\right) \left(1+{p\over 2[a](2-p)}\right)=
n \left({1\over p}-{1\over 2}\right) +\frac{n}{4[a]}.
$$  
  The proof of   Theorem~\ref{th4.5} is complete.
\end{proof}

 \medskip

 \section{Applications}
 \setcounter{equation}{0}

  As an illustration of our  results we shall  discuss some examples.
Our main results,   Theorems~\ref{th3.1} and  ~\ref{th4.5}.
  can be applied to all examples which are
 discussed in \cite{DOS}, \cite{COSY} and \cite{SYY}.
 % This  includes the fractional Schr\"odinger operators with   potentials
 % and  the Dirichlet-Neumann operators on bounded  domain.

\subsection{Symmetric  Markov chains}\
%Recall we all always assume that the doubling condition is valid.
In  \cite{A1} Alexopoulos considers  bounded symmetric Markov operator $A$ 
on a homogeneous space $X $
whose  powers $A^k$ have kernels  $A^k (x,y)$ satisfying  the following Gaussian estimates
\begin{equation} \label{eaaa}
0 \le  A^k (x,y) \le \frac{C}{V(x,{k^{1/2}})}
\exp\left(-\frac{d(x,y)^2}{k} \right)
\end{equation}
for all $k\in \N$. The  operator $I-A$ is symmetric and satisfies for every $f\in L^2(X)$,
$$
\langle(I-A)f, \, f\rangle ={1\over 2} \iint (f(x)-f(y))^2 A(x,y) d\mu(x)d\mu(y)\geq 0.
$$
Thus,  $I-A$ is positive. In addition, 
 $
\|(I-A)f\|_2\leq \|f\|_2 +\|Af\|_2\leq 2\|f\|_2.
 $
Hence, $I-A$ admits the spectral decomposition (see for example, \cite{Mc}) which allows to write  
$$
I-A=\int_0^2 (1- \lambda) dE_{A}({\lambda}).
$$
Let $F$ be a bounded Borel measurable function. Then by the spectral theorem we can define the operator
$$
F(I-A)=\int_0^2 F(1-\lambda) dE_{A}({\lambda}).
$$
Note that $I-A$ is bounded on $L^2(X)$ and $\|F(I-A)\|_{2\to 2}\leq \|F\|_{\infty}$. 
In  \cite{A1} Alexopoulos
  obtained the following spectral multiplier type 
result. In the sequel, let us consider a function $0\leq \eta\in C^\infty(\R)$  and let us assume that $\eta(t)=1$ for $t\in [1, 2]$ and that 
$\eta(t)=0$ for $t\notin  [1/2, 4]$.

\begin{theorem}\label{Thm5.1} 
	Assume that $F$ is a bounded Borel function with $\supp F \subset [0,1/2]$ and that 
	$$
	\sup_{0<t\le 1}\|\eta(\cdot)F(t\cdot )\|_{W^{n/2+\epsilon, \infty}}
	<\infty 
	$$  
	for some $\epsilon >0$. 
	Then under above assumption on the operator $A$,  the spectral multiplier 
	$F(I-A)$ extends to a bounded operator  on $L^p$ for $1 \le p \le \infty$. 
\end{theorem}

%Under the above assumption he obtained a spectral multiplier results for 
%the operator $A$. 

Our approach allows us to prove a version of Alexopoulos' result 
 under the weaker assumption of polynomial decay rather than the exponential one. We start with the following 
 statement. 

\begin{theorem}\label{Thm5.2} Let $1\leq p<2\leq q\leq \infty.$
Suppose that $(X, d, \mu)$ satisfies the doubling condition  with the doubling exponent $n$ from  \eqref{e1.2}. Assume next that $A$ is a bounded self-adjoint operator and   there exists  $k\in{\mathbb N}$ such that the kernel of 
the operator $A^k$ exists and satisfies the following estimate 
\begin{equation}\label{e8.1}
 |A^k (x,y)| \le C
\frac{1}{\max(V(x,{k^{1/2}}), {V(y,{k^{1/2}})})} \left(1+\frac{d(x,y)^2}{k} \right) ^{-N}
\end{equation}
for some    $N > n + [n/2]+1$.
Then for every $1 \le p \le \infty$,
\begin{equation}\label{e8.2}
\|F(A^k) A^k\|_{p \to p} \le C \min \left( \|F\|_{W^{s, \infty}}, \|F\|_{H^{s'}} \right)
\end{equation}
for any  $s >   n {\sigma_p} +\frac{n}{4[a]} $ and any $s' >   n {\sigma_p} + \frac{1}{2}$. 

If in addition the restriction type bounds   $\eqref{st}$ with $\tau={k^{1/2}}$ are valid on the interval $(R_1, R_2)$ 
for some $-\infty< R_1<R_2<\infty$ and supp$F\subset (R_1+\gamma, R_2-\gamma)$
for some $ \gamma >0$ then
\begin{equation}\label{e8.3}
\|F(A^k) A^k \|_{p \to p} \le C\|F\|_{W^{s, q}}
\end{equation}
$s>   n {\sigma_p} + \frac{n}{4[a]}$.
\end{theorem}

\begin{proof}  It follows from \eqref{e8.1}  that the operator $A^k$ satisfies the property $\eqref{pVEp2}$ with $a=N-n$ and 
$\tau={k^{1/2}}.$ 
%Note that for all $1\leq p<\infty,$ we have that  
%\begin{eqnarray*}
%\big\|F(A)AV^{\sigma_p}_{ {k^{1/2}}} \big\|_{p\to 2} &\leq& C  \|F(A)\|_{2\to2} \big\|AV^{\sigma_p}_{{k^{1/2}}} \big\|_{p\to 2}\\
%&\leq & C\|F\|_{\infty} 
%\end{eqnarray*}
%with $\sigma_p= {1/p}-{1/2},$  which proves  $  {\rm  ST^{\infty}_{p, 2}}(\tau) .$ Then we  apply 
%Theorem~\ref{th4.5} to obtain \eqref{e8.2} and \eqref{e8.3}. 
The theorem follows from Theorems \ref{th3.1}, \ref{th4.5} and Remark \ref{rem01}. 
\end{proof}

The following result is a direct consequence of the above theorem. 

\begin{theorem}\label{coro6.2}
Let $1\leq p<2\leq q\leq \infty.$
Suppose that $(X, d, \mu)$ satisfies the doubling condition  with the doubling exponent $n$ from \eqref{e1.2} and that $N > n + [n/2]+1${\tiny }.
Assume next that $A$ is a bounded self-adjoint operator and  that for all  $k\in{\mathbb N}$ the kernel of 
the operator $A^k$ exists and satisfies the following estimate 
\begin{equation}\label{ek}
 |A^k (x,y)| \le C
\frac{1}{\max(V(x,{k^{1/2}}), {V(y,{k^{1/2}})})} \left(1+\frac{d(x,y)^2}{k} \right) ^{-N}
\end{equation}
with the constant  $C$ independent of  $k$.
In addition we assume that the restriction type bounds   $\eqref{st} $ with $\tau={k^{1/2}}$ are valid on the interval $(R_1, R_2)\subset [-1,1]$ 
 for all $k \in \N$.  
Then for function $F$ with supp $F \subset [1/4,1/2]$
\begin{equation}\label{e55}
\|F(k(I-A))\|_{p \to p} \le C\|F\|_{W^{s,q}}
\end{equation}
for any  $s>   n {\sigma_p} +\frac{n}{4[a]}  $.
\end{theorem}
\begin{proof}
	Note that by \eqref{ek} we have that  $\|A^k\|_{2\to 2} \le C< \infty$
	for some constant $C$ independent of $k$. It follows that the spectrum of 
	$A$ is contained in the interval $[-1,1]$.
	 For a given $k\in {\mathbb N}$, we 
define a function $G$ as
$$
G(\lambda)=\frac{F(k(1-\lambda^{1/k}))}{\lambda}.
$$
and so 
$ 
G(\lambda^k)\lambda^k=F(k(1-\lambda)).
$ 
It follows from  Theorem~\ref{Thm5.2} that for $s>   n {\sigma_p} +\frac{n}{4[a]},$
$$
\|G(A^k) A^k\|_{p \to p} \le C\|G\|_{W^{s,q}},  
$$ 
which yields 
$$
\|F(k(I-A))\|_{p \to p}=\|G(A^k) A^k\|_{p \to p} \le C\|G\|_{W^{s,q}}
=C\left\|\frac{F(k(1-\lambda^{1/k}))}{\lambda}\right\|_{W^{s,q}}.
$$
Note that $\supp F\subset [1/4,1/2]$, we have
$$
\|F(k(I-A))\|_{p \to p}\le C\left\|\frac{F(k(1-\lambda^{1/k}))}{\lambda}\right\|_{W^{s,q}}\leq C\|F\|_{W^{s,q}}.
$$
This completes the proof of Theorem~\ref{coro6.2}.
\end{proof}

\begin{remark}
	1) Note that we do not assume that the operator $A$ is Markovian. \\
	%Note also that if \eqref{ek} holds just for one fixed $k\in \N$ then 
	%\eqref{e55} holds for the same $k$ and then we apply  Theorem \ref{Thm5.2}.
	2) Using  similar technique as in \cite{SYY} one  can obtain the  singular integral version of Theorem \ref{Thm5.2} stated in Theorem \ref{Thm5.1}. We do not discuss the details here.    
\end{remark}

\medskip

\subsection{Random walk on $\Z^n$} \ In this section, we consider  random walk on the $n$-dimensional integer lattice $\Z^n$.
Define the operator $A$  acting on $l^2(\Z^n)$  by the formula
\begin{equation*}
Af({\bf {d}}) = \frac1{2n}\sum_{i=1}^{n}\sum_{j=\pm 1} f({\bf {d}}+j{\bf {e}}_i)
\end{equation*}
where ${\bf {e}}_i=(0, \ldots, 1, \ldots 0 )$ and $1$ is positioned on the  $i$-coordinate.
The aim of this section is to prove the following result. Recall that  $\eta$ is the auxiliary nonzero compactly 
supported function $\eta \in C^\infty_c[1/2,4]$ as in Theorem~\ref{Thm5.1}.

%One expects that spectral decomposition of the standard Laplace operator on $\R^n$ around $0$ and of the operator $P$ around  $1$ should be
%close related. In the theorem below we show that indeed this is the case for spectral multipliers and Bochner-Riesz type analysis.

\begin{theorem}\label{th5.6} Let $1<p<\infty.$
Let $A$ be the random walk on the integer lattice defined above.
Suppose that {\rm supp}~$F\subset (0,1/n)$. Then
\begin{eqnarray}\label {e9.1}
\sup_{t>1} \|F(t(I-A)\|_{p \to p} \le C \|F\|_{H^s}
\end{eqnarray}
for any $s> n|1/p-1/2|$.

Next  we assume that   a bounded Borel
function $F: {\mathbb R_+}\to {\mathbb C}$ satisfies
{\rm supp}~$F\subset [0,1/n]$ and
 \begin{eqnarray}\label {e9.2}
 \sup_{1/n>t>0}\|\eta F(t\cdot)\|_{H^s}<\infty
\end{eqnarray}
for some $s> n|1/p-1/2|$.
Then the operator $F(I-A)$ is bounded on $L^p$ if $1<p< (2n+2)/(n+3)$ and  weak type $(1,1)$ if $p=1$.
\end{theorem}

The proof of   Theorem~\ref{th5.6} is given at the end of this section and it is based on the following restriction type estimate.

%Let $dE(\lambda)$ is the spectral measure of $P$. We have the following projection property for $dE(\lambda)$.
\begin{proposition}\label{le5.1}
Let $A$ be defined as above and $dE(\lambda)$ be the spectral measure of $A$. Then for $\lambda\in [1-1/n,1]$
\begin{eqnarray}\label{dElambda}
\|dE(\lambda)\|_{p\to p'} \leq C(1-\lambda)^{n(1/p-1/p')/2-1 }
\end{eqnarray}
 for $1<p< (2n+2)/(n+3)$.
\end{proposition}

The proof of Proposition~\ref{le5.1} is based on the following  result   due to 
 Bak and Seeger~\cite[Theorem 1.1]{BS}.
 
\begin{lemma}\label{lebs}
Consider a probability measure $\mu$ on $\mathbb{R}^n$. Assume that for positive constants 
$0<a<n$, $0<b\leq a/2$, $M_i\geq 1, i=1,2$, $\mu$ satisfies 
\begin{equation}\label{B-S1}
\sup_{r_B\leq 1} \frac{\mu(B(x_B, r_B))}{r_B^a}\leq M_1
\end{equation}
where the supremum is taken over all balls with radius $\leq 1$ and
\begin{equation}\label{B-S2}
\sup_{|\xi|\geq 1}|\xi|^b|\widehat d\mu|\leq M_2.
\end{equation}
Let $p_0=\frac{2(n-a+b)}{2(n-a)+b}$. Then
\begin{equation}\label{B-S}
\int \big|\widehat{f}\big|^2 d\mu \leq C M_1^{\frac{b}{n-a+b}}M_2^{\frac{n-a}{n-a+b}}\|f\|^2_{L^{p_0,2}(\mathbb{R}^n)},
\end{equation}
where $L^{p_0,2}$ is the Lorentz space.
\end{lemma}

\begin{proof}[Proof of Proposition~\ref{le5.1}] \
%Firstly, we manage to find the kernel of $dE(\lambda)$. 
Let ${\bf T}^n$ be the $n$-dimensional  torus (note that $n$ is equal to the homogeneous dimension of ${\bf T}^n$). For any function $f \in l^2(\Z^n)$,
one can define the Fourier series   $\mathcal{F} f \colon {\bf T}^n \to \CC$ 
of $f$ by
$$
 \mathcal{F} f({\bf {\theta}})=\sum_{{\bf {d}}\in \Z^n} f({\bf {d}}) e^{i \langle {\bf {d}} ,  {\bf {\theta}} \rangle }.
$$
Then the inverse Fourier series  $\mathcal{F}^{-1} f \colon \Z^n \to \CC$   is defined by
$$
 \mathcal{F}^{-1} f({\bf {d}})=\frac{1}{(2\pi)^n}\int_{{\bf T}^n} f({\bf {\theta}}) e^{-i \langle {\bf {d}} ,  {\bf {\theta}} \rangle}.
$$
%Then the ``Fourier transform" is an isometry from $L^2({\bf T}^n)$ onto $\ell^2(\Z^n)$   (see for example,  Grafakos~\cite[Chapter 3]{G}].
Define the convolution of $f,g\in L^2(\Z^n)$ by
$$
f\ast g({\bf {d}})=\sum_{{\bf {d}}_1\in \Z^n} f({\bf {d}}-{\bf {d}}_1)g({\bf {d}}_1).
$$
%It is easy to check that $\widehat{f\ast g}({\bf \theta})= \widehat{f}({\bf {\theta}})\cdot\widehat{ g}({\bf {\theta}})$ and
Note that
\begin{equation*}
\mathcal{F}( Af)(\theta) =       \left( \frac1n\sum_{j=1}^{n} \cos \theta_j
\right) \mathcal{F} f (\theta)=\left( G({\bf \theta})
\right) \mathcal{F} f (\theta),
\end{equation*}
where
$$
G(\theta)= \frac1n\sum_{j=1}^{n} \cos \theta_j .
$$
%We define
%$$
%F(P)f({\bf d})=\mathcal{F}^{-1}\left(F(G(\theta))\widehat f(\theta)\right)({\bf d})
%$$
%for all bounded Borel function $F$. We can easily testify this is a functional calculus of $P$, that is, satisfy the conditions
%(a)--(d) in Reed and Simmon~\cite[Theorem VIII.5]{RS}. According to \cite[Theorem VIII.5]{RS}, this functional calculus is unique. So we surely have
%$$
%F(P)f({\bf d})=\mathcal{F}^{-1}\left(F(G(\theta))\widehat f(\theta)\right)({\bf d})=\int_{\R} F(\lambda)dE(\lambda)f.
%$$
%Now we choose a local patch of coordinates $\theta, \lambda$ in such a way that $\theta,\lambda \in \sigma_\lambda$ for
%all $\theta, \lambda$. Note that Jacobian in this coordinates is $\frac{1}{|\nabla G|}\ d\sigma_\lambda(\theta)$.
%Hence 
Hence for any  continuous function $F$
\begin{eqnarray*}
\int_{-1}^1 F(\lambda)dE(\lambda)f= F(A)f({\bf d})&=&\mathcal{F}^{-1}\left(F(G(\theta))\widehat f(\theta)\right)\\
%&=& \mathcal{F}^{-1}\left(F(G(\theta))\right)\ast f\\
&=& \left(\frac{1}{(2\pi )^n}\int_{\R} \int_{\sigma_\lambda}F(\lambda) e^{i\langle {\bf {d}},  {\bf {\theta}} \rangle} \frac{1}{|\nabla G|}\ d\sigma_\lambda(\theta)d\lambda\right)\ast f\\
&=& \int_{\R} F(\lambda) \left(\Gamma_\lambda\ast f \right)d\lambda,
\end{eqnarray*}
where $\sigma_\lambda $ is the level set defined  by the formula
$$
\sigma_\lambda = \{\theta \colon \,     \frac1n\sum_{j=1}^{n} \cos \theta_j =\lambda\}        \subset {\bf T}^n
$$
and
\begin{eqnarray*}%\label{e1.6}
	\Gamma_\lambda( {\bf {d}}) =\frac{1}{(2\pi)^n}
	\int_{\sigma_\lambda}e^{i\langle {\bf {d}} ,  {\bf {\theta}} \rangle }\frac{1}{|\nabla G|}\ d\sigma_\lambda(\theta)
\end{eqnarray*}
 for all $\lambda\in [-1,1]$. Thus 
$$
dE(\lambda) f = f *\Gamma_\lambda.
$$
%{\color{red}
%Peng: Actually if my derivation is right, I lost a ``$i^n$" comparing to the formula of $\Gamma_\lambda$. Why?
%} \\
%Secondly, we manage to express $dE(\lambda)$ by the ``restriction" type operator.
For the range $\lambda \in (1-1/n,1)$ considered in the proposition, changing variable yields 
\begin{eqnarray*}%\label{e1.6}
	\Gamma_\lambda( {\bf {d}}) =\frac{({1-\lambda})^{(n-1)/2}}{(2\pi i)^n}
	\int_{\sigma'_\lambda}e^{i \langle {\bf {d}},
		\,  (1-\lambda)^{1/2}{\bf {\theta}} \rangle }\frac{1}{|\nabla G|((1-\lambda)^{1/2}\theta)}\ d{\widetilde \sigma}_\lambda(\theta),
\end{eqnarray*}
where
$$
{\widetilde \sigma}_\lambda = \{\theta \colon \,     \frac1n\sum_{j=1}^{n} \cos ((1-\lambda)^{1/2}\theta_j ) =\lambda\}        \subset {\bf T}^n.
$$
%Then the kernel $\Gamma_\lambda$ of $dE(\lambda)$ can be written as
%\begin{eqnarray*}%\label{e1.6}
%\Gamma_\lambda( {\bf {d}}) =\frac{A(\lambda)((1-\lambda)^{1/2})^{n-1}}{(2\pi i)^n}
%\int_{{\widetilde \sigma}_\lambda}e^{i(1-\lambda)^{1/2}{\bf {d}}\cdot\theta}\frac{1}{|\nabla G|(\theta(1-\lambda)^{1/2})A(\lambda)}\ d{\widetilde \sigma}_\lambda(\theta),
%\end{eqnarray*}
Next we define a probability measure $\mu_\lambda$ on the surface ${\widetilde \sigma}_\lambda$ by the formula 
$$
d\mu_\lambda=\frac{1}{|\nabla G|((1-\lambda)^{1/2}\theta)N(\lambda)}\ d{\widetilde \sigma}_\lambda(\theta),
$$
where
$$
N(\lambda)=\int_{{\widetilde \sigma}_\lambda} \frac{1}{|\nabla G|((1-\lambda)^{1/2}\theta)}\ d{\widetilde \sigma}_\lambda(\theta).
$$
For $\lambda \in (1-1/n,1)$ we define  the restriction type operator 
$R_\lambda \colon \, \ell^1(\Z^n) \to L^2({\widetilde \sigma}_\lambda,\mu_\lambda)$ by
$$
(R_\lambda f)(\theta)=
 \mathcal{F} f({\bf {\theta}}(1-\lambda)^{1/2})=\sum_{\bf {d} \in \Z^n} f({\bf {d}}) e^{i \langle {\bf {d}},
  \,  (1-\lambda)^{1/2}{\bf {\theta}} \rangle },\quad\quad \theta\in {\widetilde \sigma}_\lambda. 
$$
Then   the dual operator $R_\lambda^*$ is given by 
$$
(R_\lambda^* f)({\bf {d}})=
 \int_{{\widetilde \sigma}_\lambda}e^{-i(1-\lambda)^{1/2}\langle {\bf {d}},  
 {\bf {\theta}} \rangle }\frac{f(\theta)}{|\nabla G|((1-\lambda)^{1/2}\theta)N(\lambda)}\ d{\widetilde \sigma}_\lambda(\theta).
$$
Hence 
$$
dE(\lambda) f=\frac{N(\lambda)((1-\lambda)^{1/2})^{n-1}}{(2\pi )^n} R_\lambda^*R_\lambda f.
$$
Following the standard approach on the Euclidean space we study boundedness of
the operator $R_\lambda^*$ acting from   $L^2({\widetilde \sigma}_\lambda,\mu_\lambda)$ to 
$ \ell^p(\Z^n)$.
%Set $\tilde{G}(\theta)= {G}((1-\lambda)^{1/2}\theta) = \frac1n\sum_{j=1}^{n} \cos ((1-\lambda)^{1/2}\theta_j) $.
%Since $\lambda \in (1-1/n, 1)$, $\tilde{G}(\theta)=\lambda$ holds only if $\cos ((1-\lambda)^{1/2}\theta_j)>0$. Then
%$$
%(1-\lambda)^{1/2}\theta\in [-\pi/2,\pi/2]^n, \quad\quad \theta\in {\widetilde \sigma}_\lambda,
%$$
%which implies that $R_\lambda^* f(\xi)$ can be extended to an analytic function of exponential type less than $1$. 
Next we define the operator $\tilde{R}_\lambda^* \colon L^2({\widetilde \sigma}_\lambda,\mu_\lambda) \to L^\infty(\R^n)$ by 
$$
(\tilde{R}_\lambda^* f)({ \xi})=
\int_{{\widetilde \sigma}_\lambda}e^{-i\langle {\xi} ,  (1-\lambda)^{1/2}{\bf {\theta}} \rangle }d \mu_\lambda(\theta).
$$
By the Plancherel-P\'olya inequality (cf. \cite[Section 1.3.3]{Tri})
\begin{equation}\label{PP}
\|(R_\lambda^* f)({\bf {d}})\|_{L^p(\Z^n)}\sim \|R_\lambda^* f(\xi)\|_{L^p(\R^n)}
\end{equation}
for all $1\leq p\leq \infty$.
Hence, it suffices to study  $R_\lambda^* f(\xi),\xi\in \R^n$.
%\begin{lemma}
%For $\xi\in \R^n$, define
%$$
%(R_\lambda^* f)(\xi)=
% \int_{{\widetilde \sigma}_\lambda}e^{-i(1-\lambda)^{1/2}{\xi}\cdot\theta}\frac{f(\theta)}{|\nabla G|((1-\lambda)^{1/2}\theta)A(\lambda)}\ d{\widetilde \sigma}_\lambda(\theta)
%$$
%where ${\widetilde \sigma}_\lambda$, $G$ and $A$ are defined as above. Then for $1\leq p\leq \frac{2n+2}{n+3}$
%$$
%\|R_\lambda^* f(\xi)\|_{L^{p'}(\R^n)}\leq C(1-\lambda)^{1/2}^{-\frac{n}{p'}}\|f\|_{L^2(\sigma_\lambda', d\mu)}.
%$$
%\end{lemma}
%\begin{proof}
Set $\tilde{G}(\theta)= \frac1n\sum_{j=1}^{n} \cos ((1-\lambda)^{1/2}\theta_j) $. Denote $H(\tilde{G})$  the Hessian corresponding to $\tilde{G}$. Then the Gaussian curvature for an implicitly defined surface corresponding to the equation $\tilde{G}(\theta)=\lambda$
is given by the following formula

\begin{eqnarray*}\label{quuu}
K&=& -\left| \begin{array}{ll}
H(\tilde{G}) & \nabla \tilde{G}^T\\
\nabla \tilde{G}   &0
\end{array}
\right||\nabla \tilde{G}|^{-{(n+1)}}\\
&=&\frac{(-1)^{n+1}((1-\lambda)^{1/2})^{n-1}\prod\limits_{j=1}^n \cos ((1-\lambda)^{1/2}\theta_j )\left(\sum\limits_{j=1}^n 
\sin ( \theta_j (1-\lambda)^{1/2} ) \tan ((1-\lambda)^{1/2}\theta_j) \right)}{ \left(\sum\limits_{j=1}^n \sin^2 ((1-\lambda)^{1/2}\theta_j)\right)^{{ (n+1)/2}}}.
\end{eqnarray*}
Note that if $\lambda \in (1-1/n,1)$, then  $\cos ((1-\lambda)^{1/2}\theta_j)>0$ for all $j$ and 
$$\cos ((1-\lambda)^{1/2}\theta_j)\geq 1-(1-\lambda)n.$$ Indeed, otherwise
$$
\frac1n\sum_{j=1}^{n} \cos ((1-\lambda)^{1/2}\theta_j) < \frac{n-1+1-(1-\lambda)n}{n}=\lambda
$$
which contradicts $\theta\in {\widetilde \sigma}_\lambda$. 
It follows that for every $n\ge 2$ there  exists a positive  constant
$C_n >0$ which  does not depend on $\lambda$ and $\theta$ such that 
$$
|K|\geq C_n.
$$
 There exists also a constant $C>0$ such that for all $\lambda \in (1-1/n,1)$
$$
(1-\lambda)^{1/2}\geq |\nabla G|((1-\lambda)^{1/2}\theta)=\frac{1}{n}\left(\sum_{j=1}^n \sin^2 ( (1-\lambda)^{1/2}\theta_j)\right)^{{ 1/2}}\geq C (1-\lambda)^{1/2}
$$
so $N(\lambda)\sim {(1-\lambda)}^{-1/2}$. Then from Stein~\cite[Page 360, Section 5.7 of Chapter VIII]{St2}, we know that
\begin{eqnarray}\label{surfacemeasure}
|\widehat{d\mu_\lambda}|\le C (1+|\xi|)^{(1-n)/2},
\end{eqnarray}
where $C$ just depends on $n$ and does not depend on $\lambda$ and $\theta$.

Now, it is not difficult to check that surfaces ${\widetilde \sigma}_\lambda$ and measures $\mu_\lambda$ satisfy assumptions of
Lemma~\ref{lebs}. The required exponent for  \eqref{B-S1} is equal to  $a=n-1$. In addition  $d\mu_\lambda$ satisfies \eqref{B-S2} with $b=(n-1)/2$
uniformly in $\lambda \in (1-1/n,1)$.
Hence by Lemma \ref{lebs}
$$
\int_{{\widetilde \sigma}_\lambda} \big|\widehat{f}\big|^2 d\mu_\lambda \leq C\|f\|^2_{L^{p}(\R^n)}
$$
%for $1\leq p\leq  p_0=\frac{2n+2}{n+3}$. By duality for $f\in L^2(\sigma_\lambda', d\mu)$
%$$
%\left\|\int f(\theta)e^{-i \langle \xi, \,  \theta\rangle} d\mu(\theta)\right\|_{p'}
%\leq C\|f\|_{L^2(\sigma_\lambda', d\mu)}.
%$$
%Note that
%$$
%\tilde{R}_\lambda^* f(\xi)=\int f(\theta)e^{-i (1-\lambda)^{1/2}\langle \xi, \,  \theta\rangle } d\mu(\theta).
%$$
Hence
$$
\|\tilde{R}_\lambda^* f(\xi)\|_{L^{p'}(\R^n)}\leq C{(1-\lambda)}^{-\frac{n}{2p'}}\|f\|_{L^2({\widetilde \sigma}_\lambda, \, \mu_{\lambda})}.
$$
%$$
%\|\tilde{R}_\lambda^* f(\xi)\|_{L^{p'}(\R^n)}={(1-\lambda)}^{-\frac{n}{2p'}}
%\left\|\int f(\theta)e^{-i \langle \xi, \,  \theta\rangle} d\mu(\theta)\right\|_{p'}\leq C{(1-\lambda)}^{-\frac{n}{2p'}}\|f\|_{L^2(\sigma_\lambda', \, d\mu)}.
%$$
Thus by the Plancherel-P\'olya inequality \eqref{PP}
$$
\|R_\lambda^* f({\bf{d}})\|_{L^{p'}(\Z^n)}\leq C{(1-\lambda)}^{-\frac{n}{2p'}}\|f\|_{L^2({\widetilde \sigma}_\lambda, \, \mu_{\lambda})}
$$
for $1\leq p\leq \frac{2n+2}{n+3}$.
By duality 
$$
\|dE(\lambda)\|_{p\to p'} \leq \frac{N(\lambda){(1-\lambda)}^{\frac{n-1}{2}}}{(2\pi )^n}\| R_\lambda^*R_\lambda\|_{p\to p'}\leq C(1-\lambda)^{n(1/p-1/p')/2-1 },
$$
for $1\leq p\leq \frac{2n+2}{n+3}$ and $\lambda \in (1-1/n,1)$.
This completes the proof of Proposition~\ref{le5.1}.
\end{proof}

\medskip

Now we are able to conclude the proof of  Theorem~\ref{th5.6}.

 \medskip
 
\begin{proof}[Proof of Theorem~\ref{th5.6}]
For every $k\in {\mathbb N}$, we denote $A^k({\bf {d}}_1,{\bf {d}}_2)$ the kernel of $A^k$ for $k\in \mathbb{Z}$. Note that $V(x, k) \sim k^n$.
 It is well-known  (see e.g. \cite{HS}) that 
$A^k({\bf {d}}_1,{\bf {d}}_2)$ satisfies the following Gaussian type upper estimate:
\begin{eqnarray}\label{HeatofRadomW}
A^k({\bf {d}}_1,{\bf {d}}_2)\leq Ck^{-n/2} \exp\left( -\frac{|{\bf {d}}_1-{\bf {d}}_2|^2}{ck}\right).
\end{eqnarray}
%From~\cite{HS}, we see that $P^k$ satisfies   satisfies 
%\eqref{pVEp2} with $p=1$ for every $a>0$ and  $\tau={k^{1/2}}$ for every $k\in{\mathbb N}$.

Next we  verify that the operators $A^k$ satisfy  condition $ ({{\rm  ST^2_{p, 2}}(\tau) })$ with $\tau={k^{1/2}}$ 
uniformly for all $k\in \N$ for all bounded Borel functions $F$ such that {\rm supp}~$F\subset (1-1/n,1)$. 
%Noting  that $V(x, k) \sim k^n$, we need to show  that for all $F$ with {\rm supp}~$F\subset (1-1/n,1)$,  
%\begin{equation}\label{k}
%\big\|F(P^k)P^k  \big\|_{p\to 2} \leq C k^{-n(1/p-1/2)/2}\|F\|_{L^2} 
%\end{equation}
% and this can be obtained by using Proposition~\ref{le5.1}. 
By  $T^*T$ argument and Proposition~\ref{le5.1},
  \begin{eqnarray*}
\big\| F(A^k)A^k\big\|_{p\to 2}^2 = \big\||F|^2(A^k)A^{2k}\big\|_{p\to p' }
&\le& \int_{ [(n-1)/n]^{1/k} }^1 |F|^2( \lambda^k) \lambda^{2k}\big\|dE_{A}(\lambda)\big\|_{p \to p'}\, d\lambda \nonumber \\
&\le&{C} \int_{(n-1)/n}^{1} |F|^2( \lambda) \lambda^{2}(1-\lambda^{1/k})^{n(1/p-1/p')/2-1} \,   d\lambda^{1/k}\\
%&\le& {C} \int_{(n-1)/n}^{1} |F|^2( \lambda) \lambda^{2}\left({1\over k}\right)^{n(1/p-1/p')/2-1}\left({1\over k}\right)\lambda^{1/k-1} \,   d\lambda\\
&\le&C \left({1\over k}\right)^{n(1/p-1/p')/2}\|F\|_2,
\end{eqnarray*}
as desired. 

Now  \eqref{e9.1} follows from Theorem \ref{coro6.2}.
% By Theorem~\ref{th4.5}, we have for function $G$ with $\supp G\subset (1-1/n,1)$
%$$
%\|G(P^k)P^k\|_{p\to p}\le C\|G\|_{H^s}
%$$
%Define a function $G$ by
%$$
%G(\lambda)=\frac{F(k(1-\lambda^{1/k}))}{\lambda}.
%$$
%We have that
% $
%G(\lambda^k)\lambda^k=F(k(1-\lambda)),
% $
%and so 
%$$
%\|F(k(I-P))\|_{p \to p}=\|G(P^k) P^k\|_{p \to p} \le \|G\|_{H^s}
%=\|\frac{F(k(1-\lambda^{1/k}))}{\lambda}\|_{H^s}.
%$$
%Note that $\supp F\subset [0,1/n]$, we have
%$$
%\|F(k(I-P))\|_{p \to p}\le \|\frac{F(k(1-\lambda^{1/k}))}{\lambda}\|_{H^s}\leq C\|F\|_{H^s}
%$$
%where $C$ is  a constant independent of $k$.
The $L^p$ boundedness of $F(I-A)$ for functions $F$ satisfying condition 
\eqref{e9.2} follows from \cite[Theorem 3.3]{SYY}. 
This completes the proof of  Theorem~\ref{th5.6}.
%If function $F$ saitsfies~\eqref{e1.3}, noting that $P^n$ satisfies Gaussian decay and then following similar argument as the proof of Sikora-Yan-Yao~\cite[Theorem 3.1]{SYY}, we have that  $F(I-P)$ is bounded on $L^p$ follows from the first conclusion of this theorem.
\end{proof}

\begin{remark}
For $n=2$ it is enough to assume that {\rm supp}~$F\subset (0,1)$.
\end{remark}

\begin{remark}
	There is another approach to the proof of  Theorem~\ref{th5.6} via transference type statements on equivalence of $L^p$ boundedness of the Fourier 
	integral and the Fourier series multipliers under suitable condition on
	the multiplier support (cf. \cite{xs}). 
	% We will not discuss details of 
	%this approach here, see \cite{xs} 
\end{remark}

\medskip

\subsection{Fractional Schr\"odinger operators}
Let $n\geq 1$ and $V, W$ be  locally integrable non-negative functions on
${\Bbb R}^n$. \\
Consider the fractional Schr\"odinger operator  with a potentials  $V$ and $W$:
$$
L=(-\Delta + W)^{\alpha} +V(x), \ \  \ \alpha\in (0,1].
$$
The particular case $\alpha = \frac{1}{2}$ is often referred to as the relativistic Schr\"odinger operator.
The operator $L$ is self-adjoint as an operator associated with a well defined closed quadratic form. By the classical subordination formula together with the Feynman-Kac formula it follows that  the semigroup kernel
$p_t(x,y)$ associated to $e^{-tL}$ satisfies the estimate
\begin{eqnarray*}
	0\leq  p_t(x,y) \leq Ct^{-{n\over 2\alpha}} \Big(1+t^{-{1\over 2\alpha}}|x-y|\Big)^{-(n+2\alpha)}
\end{eqnarray*}
for all $t>0$ and $x,y\in{\Bbb R}^n$.
See page 195 of \cite{Ou}.
Hence,  estimates \eqref{pVEp2}  hold  for  $p=1$ and  $a=2\alpha$. If $n = 1$ and $\alpha > \frac{1}{2}$ then we can apply Corollary \ref{coro3.2} and obtain a spectral multiplier result for $L$.

 \bigskip

 \noindent
{\bf Acknowledgments.}    P. Chen was supported by NNSF of China 11501583, Guangdong Natural Science Foundation
2016A030313351. The research of E.M. Ouhabaz is partly supported by the ANR project RAGE ANR-18-CE40-0012-01.
 A.~Sikora was supported by
Australian Research Council  Discovery Grant  DP160100941.
 L. Yan was supported by the NNSF
of China, Grant No. ~11871480, and Guangdong Special Support Program. 
The authors would like to  thank Xianghong Chen for useful discussions on the proof of Theorem \ref{th5.6}.

\end{document}